\documentclass[letterpaper,11pt]{amsart}

\usepackage{etex}
\usepackage{ifdraft}

\usepackage[utf8]{inputenc}
\usepackage[T1]{fontenc}

\usepackage{lmodern}
\usepackage{fourier}

\usepackage{amssymb, amsfonts}

\usepackage{mathrsfs} 
\usepackage{dsfont}
\usepackage[usenames,dvipsnames]{xcolor}
\usepackage{xcolor}
\usepackage{url}

\usepackage{stmaryrd}
\usepackage{mathrsfs}

\usepackage[english]{babel}

\usepackage[margin=7em]{geometry}
\usepackage{needspace}

\usepackage{amsmath}

\usepackage[textsize=tiny,colorinlistoftodos]{todonotes}

\usepackage{booktabs}
\usepackage[inline]{enumitem}
\usepackage{float}

\usepackage{tikz}
\usepackage{tikz-cd}
\usetikzlibrary{matrix,arrows,decorations.pathmorphing}

\usepackage{mathtools}

\usepackage{hyperref}
\hypersetup{colorlinks=true,citecolor=blue!70!black,linkcolor=red!60!black,linktocpage=true}
\usepackage[capitalise,nameinlink]{cleveref}
\crefname{projection}{Projection}{Projections}
\crefname{exactsequence}{Exact Sequence}{Exact Sequences}
\crefname{map}{Map}{Maps}

\numberwithin{equation}{section}
\newtheorem{thm}[equation]{Theorem} 
\crefname{thm}{Theorem}{Theorems}
\newtheorem{theorem}[equation]{Theorem} 
\newtheorem{prop}[equation]{Proposition}
\newtheorem{proposition}[equation]{Proposition}
\crefname{prop}{Proposition}{Propositions}
\newtheorem{lemma}[equation]{Lemma} 
\newtheorem{cor}[equation]{Corollary}
\crefname{cor}{Corollary}{Corollaries}

\crefname{corollary}{Corollary}{Corollaries}
\newtheorem{example}[equation]{Example}
\newtheorem{remark}[equation]{Remark}
\newtheorem{definition}[equation]{Definition}

\let\realsubsection\subsection
\def\subsection{\setcounter{subsection}{\arabic{equation}}
   \refstepcounter{equation}%
   \realsubsection
}
\crefname{subsection}{Subsection}{Subsections}

\newenvironment{enumerateroman}{
\begin{enumerate}[label=(\roman*), leftmargin=0pt,labelindent=2em,itemindent=!]
}{
\end{enumerate}
}

\newenvironment{enumeratearabic*}{
\begin{enumerate*}[label=(\arabic*)] 
}{
\end{enumerate*}
}

\newenvironment{enumerateroman*}{
\begin{enumerate*}[label=(\roman*)] 
}{
\end{enumerate*}
}

\newcommand{\nbd}{\nobreakdash-\hspace{0pt}}

\newcommand{\fref}[2]{\hyperref[#2]{#1~\ref*{#2}}}

\makeatletter
\newcommand{\writelabel}[1]{#1\def\@currentlabel{#1}}
\makeatother

\newcommand{\minwidthmathbox}[2]{%
  \mathmakebox[{\ifdim#1<\width\width\else#1\fi}]{#2}%
}

\newcommand{\bbL}{\mathbb{L}}

\newcommand{\cC}{\mathcal{C}}

\newcommand{\cO}{\mathcal{O}}

\newcommand{\frakg}{\mathfrak{g}}

\newcommand{\frakk}{\mathfrak{k}}

\newcommand{\frakm}{\mathfrak{m}}

\newcommand{\rml}{\mathrm{l}}

\newcommand{\rmr}{\mathrm{r}}

\newcommand{\rmC}{\mathrm{C}}
\newcommand{\rmD}{\mathrm{D}}

\newcommand{\rmJ}{\mathrm{J}}

\newcommand{\rmL}{\mathrm{L}}

\newcommand{\rmO}{\mathrm{O}}

\newcommand{\rmR}{\mathrm{R}}

\newcommand{\td}{\tilde}

\newcommand{\ov}{\overline}

\newcommand{\mrelspace}[1]{\mathrel{\mspace{#1}}}

\NewCommandCopy{\rightarroworig}{\rightarrow}
\renewcommand{\rightarrow}
{\protect\relbar\mrelspace{-9.7mu}\rightarroworig}

\NewCommandCopy{\leftarroworig}{\leftarrow}
\renewcommand{\leftarrow}
{\protect\leftarroworig\mrelspace{-9.7mu}\relbar}

\renewcommand{\longrightarrow}
{\protect\relbar\mrelspace{-3.2mu}\relbar\mrelspace{-9.5mu}\rightarroworig}
\newcommand{\longhookrightarrow}
{\protect\lhook\mrelspace{-3.1mu}\relbar\mrelspace{-3.2mu}\relbar\mrelspace{-11.7mu}\rightarroworig}
\newcommand{\longtwoheadrightarrow}
{\protect\relbar\mrelspace{-3mu}\rightarroworig\mrelspace{-15mu}\rightarroworig}

\renewcommand{\Im}{\mathrm{Im}}

\NewCommandCopy{\slashorig}{\slash}
\renewcommand{\slash}{\mathbin{\slashorig}}

\newenvironment{smatrix}{\begin{smallmatrix}}{\end{smallmatrix}}

\newcommand{\ZZ}{\mathbb{Z}}
\newcommand{\QQ}{\mathbb{Q}}
\newcommand{\RR}{\mathbb{R}}
\newcommand{\CC}{\mathbb{C}}

\newcommand{\SL}[1]{\mathrm{SL}_{#1}}

\renewcommand{\det}{\mathrm{det}}

\newcommand{\HS}{\mathbb{H}}

\newcommand{\discolorlinks}[1]{{\hypersetup{hidelinks}#1}}

\newcommand{\llb}{[\hspace{-.23ex}[}
\newcommand{\rrb}{]\hspace{-.23ex}]}
\newcommand{\llp}{(\hspace{-.23ex}(}
\newcommand{\rrp}{)\hspace{-.23ex})}
\newcommand{\ixt}{{\mathbf t}}

\newcommand{\D}{\mathcal{D}}

\newcommand{\ot}{\otimes}

\newcommand{\HJ}{{\mathbb H}_{\rmJ}}
\newcommand{\CHJ}{C^{\infty}(\HJ)}
\newcommand{\OHJ}{\cO(\HJ)}
\newcommand{\holos}{\cO(\HJ)}
\newcommand{\DOT}{\bullet}

\newcommand{\del}{\partial}
\newcommand{\ddel}{ \boldsymbol{\del} }
\newcommand{\delz}{ \boldsymbol{\del}_{\! z}}

\newcommand{\PDO}{\ensuremath{\Psi\rmD\rmO}}
\newcommand{\po}{\PDO\big(\ddel, \CHJ \big)}
\newcommand{\poh}{\PDO\big(\ddel, \OHJ \big)}

\newcommand{\ml}{\ensuremath{m_\rml}}
\newcommand{\mr}{\ensuremath{m_\rmr}}
\newcommand{\mlr}{\ensuremath{{(\ml, \mr)}}}
\newcommand{\Lmod}{{\rmL}^{}}
\newcommand{\Lell}{{\rmL}^{J}}
\newcommand{\Rmod}{{\rmR}^{}}
\newcommand{\Rell}{{\rmR}^{J}}
\newcommand{\Cmod}{{\rmC}^{}}
\newcommand{\Cell}{{\rmC}^{J}}

\newcommand{\GJ}{\ensuremath{{G}}}

\newcommand{\frakkJ}{\ensuremath{\frakk}}

\newcommand{\tx}[1]{\text{#1}}

\newcommand{\frakmJ}{\ensuremath{\frakm}}
\newcommand{\frakgJ}{\ensuremath{\frakg}}

\newcommand{\tree}{\Upsilon}

\newcommand{\Gr}{\ensuremath{\mathrm{gr}}\,}

\def\smath#1{\text{\scalebox{.9}{$#1$}}}
\def\ssmath#1{\text{\scalebox{.75}{$#1$}}}
\def\sfrac#1#2{\smath{\frac{#1}{#2}}}
\def\sbinom#1#2{\ssmath{\binom{#1}{#2}}}
\begin{document}
\nocite{*}

\title[Pseudodifferential Jacobi forms and
Geometric Rankin-Cohen Brackets]{Pseudodifferential
Jacobi forms and\\Geometric Rankin-Cohen Brackets}
\author{Martin Raum}
\address{Department of Mathematical Sciences,
Chalmers University of Technology and University of Gothenburg,
SE-412 96 Gothenburg, Sweden}
\email{martin@raum-brothers.eu}
\urladdr{https://martin.raum-brothers.eu}
\author{Anne V.\ Shepler}
\address{Department of Mathematics, University of North Texas,
Denton, Texas 76203, USA}
\email{ashepler@unt.edu}
\urladdr{https://sites.math.unt.edu/~ashepler/}
\thanks{The first author was supported by Vetenskapsr\aa det Grant~2023-04217. 
The second author was partially supported Simons Foundation 
grant 949953.}

\subjclass{
Primary 11F50, 
11F60; Secondary 
16W70, 32W25 
}
\keywords{Jacobi forms, automorphic forms,
Rankin-Cohen brackets}

\begin{abstract}
Cohen, Manin, and Zagier recovered the Rankin-Cohen bracket for modular forms
from an action of the modular group on pseudodifferential operators
whose coefficients are holomorphic functions on the Poincar\'e upper half plane.
We investigate pseudodifferential operators on the Jacobi upper half space
with respect 
to the elliptic variable
instead of the modular variable
typically considered.
We introduce a family of actions of the Jacobi group and show that
a space of invariant pseudodifferential operators is isomorphic to the space
of Jacobi forms by producing an equivariant map.  Our construction arises from 
the explicit action of a Casimir operator for the complexified Lie algebra of the real Jacobi Lie group.
As an application, 
we identify new 
families of Rankin-Cohen brackets
with geometric origin
indexed by a complex parameter.
In particular, we isolate 
a
subvariety of lines of Rankin-Cohen brackets in each degree
of expected dimension $1$
reflecting the geometry
of the Jacobi upper half space.
\end{abstract}

\date{June 29, 2026.}

\maketitle

\setcounter{tocdepth}{1}
\tableofcontents


\section{Introduction}

In 1997, Cohen, Manin, and Zagier~\cite{CMZ} 
connected
Rankin-Cohen brackets to the noncommutative
multiplication of 
certain pseudodifferential operators.
Various authors have considered pseudodifferential operators in special cases with the
intent of connecting elliptic modular
forms to the deformation quantization
of Poisson manifolds
advertised by Kontsevich~\cite{Kontsevich}.
Consider
Bieliavsky, Tang, and Yao~\cite{BieliavskyTangYao},
Dumas and Royer~\cite{DumasRoyer},
and Ovsienko~\cite{Ovsienko},
for example.
One asks if this approach 
may yield
canonical families
of Rankin-Cohen brackets
reflecting inherent geometry
for other modular and automorphic forms.

To address this question, we turn to an approach based on equivariant operators on graded noncommutative algebras advertised by Beilinson to Cohen, Manin, and Zagier.
Specifically, Beilinson suggested using a Casimir operator to link pseudodifferential operators to modular forms, see~\cite[p.~23]{CMZ}.
This idea previously has not 
been pursued with rigor in the literature.
In particular, we identify a non-vanishing condition which reflects the reducibility of certain Harish-Chandra modules.
This condition is required to obtain an equivariant direct sum decomposition of the space of pseudodifferential operators. In the case of elliptic modular forms, this may 
be derived via formal combinatorial arguments,
see~\cite{CMZ}.
In cases where this vanishing condition fails, indecomposable constituents akin to the quasi-modular form~$E_2$ can appear, explaining in part the focus of Dumas and Royer~\cite{DumasRoyer}
on quasi-modular forms as opposed to modular forms.
The resulting decomposition result applies to other settings beyond elliptic modular forms in which formal combinatorial proofs seem prohibitively intricate.
We examine in detail the case of Jacobi forms for its historical prominence, importance in combinatorial number theory, and a subtle non-uniqueness in its theory of Rankin-Cohen brackets.
While we focus on 
Jacobi forms, our theory is conducive to studying pseudodifferential operators or Taylor expansions (see Choie and Lee~\cite{ChoieLee,ChoieLeeBook}) in a general setting.
This approach should allow for 
an extension of results here  to weakly holomorphic forms, quasi-modular forms, and forms with singularities at torsion points, for example.

A natural question arises
for Jacobi forms:
What notion
of Rankin-Cohen bracket
reflects key properties?
For elliptic
modular forms, 
the only ambiguity in defining
Rankin-Cohen brackets arises from
rescaling.
In contrast, for Jacobi forms,
B\"ocherer~\cite{Boecherer} 
showed that the space of
(all possible)
$\nu$-th Rankin-Cohen brackets
grows 
in dimension as $\nu$ grows
(see \cref{Rankin-CohenBracketConstruction}).
Choie~\cite{Choie}
described a $1$-dimensional
subspace for each $\nu$
given by
polynomials in the heat operator
(see~\cite{Boecherer}).
Choie and Eholzer~\cite{ChoieEholzer}
gave a basis
for the space of all Rankin-Cohen brackets
by further leveraging the heat operator,
which via the theta decomposition reflects the action on elliptic modular forms.
Nevertheless, it is not immediately clear
how to define 
privileged families of brackets
reflecting 
intrinsic
traits of Jacobi forms.

Choie, Dumas, Martin, and Royer  
\cite{ChoieDumasMartinRoyer2017,ChoieDumasMartinRoyer} 
found families of Rankin-Cohen brackets for {\em weak}
Jacobi forms
using the theory of Connes and
Moscovici~\cite{ConnesMoscovici}, which
leverages Hopf symmetries.
Weak Jacobi forms
lend themselves to an extension
of methods used for elliptic modular forms
via their Taylor expansions.
Note that Dumas and Martin
\cite{DumasMartin}
considered pseudodifferential operators
on $\mathbb{H}$ for examining Rankin-Cohen brackets
of modular forms using
Jacobi-like forms, 
which exhibit invariance only under
$\text{SL}_2(\ZZ)$ as opposed to the full
Jacobi group.
Also see Choie and Lee~\cite{ChoieLee}
and~\cite{ChoieLeeBook}.

We define here families of Rankin-Cohen brackets for Jacobi forms respecting the geometry of the Jacobi upper half space
$\HJ$.
We leverage the noncommutative
multiplication of pseudodifferential
operators 
to identify $1$-complex-parameter's
worth of families of Rankin-Cohen brackets
of geometric origin.
This gives in each degree $\nu$
a
subvariety of lines
of expected dimension $1$
in the space of
all possible Rankin-Cohen brackets
as described by B\"ocherer~\cite{Boecherer}.
Our brackets also capture the inherent growth condition of Jacobi forms.
The difference between weak and non-weak Jacobi forms is reflected 
in the support of the Fourier expansion.
Methods from elliptic modular forms
prove fruitful in the weak case, 
but for some applications in combinatorial
number theory,
this growth condition is crucial.
Our brackets may serve as a useful
tool in such settings.

\subsection{Cohen-Manin-Zagier}
Cohen, Manin, and Zagier~\cite{CMZ} 
considered the natural noncommutative multiplication
on a space of pseudodifferential operators 
induced from the Leibniz rule,
see also Zagier~\cite{Indian}.
Specifically, they considered the space
$$\Psi\rm{DO}(\del_{\tau}, C^{\infty}(\HS))
\qquad\text{ 
of pseudodifferential operators in 
$\del_\tau:={\del}/{\del \tau}$}\, 
$$
whose coefficients are smooth functions 
$\HS\rightarrow\CC$ on the
Poincar\'e upper half plane
$\HS$ in the variable $\tau$.
Their work hinged on
an {\em equivariant mapping}
\begin{equation}
\label{CMZ-map}
C^{\infty}(\HS)
\ 
\xleftrightarrow{\ \ \SL{2}(\RR)\text{-equivariant }\ \ }\ 
\PDO(\del_{\tau}, C^{\infty}(\HS))
\, .
\end{equation}
They carried the noncommutative multiplication
of pseudodifferential operators
over to $C^{\infty}(\HS)$ 
to define two parameters' worth of 
(associative) noncommutative multiplications
on $C^{\infty}(\HS)$ and thus on modular forms. They then showed how
to recover the Rankin-Cohen bracket for modular forms
from such multiplications
and gave combinatorial conversion formulas
of independent interest.
Their work describes the Rankin-Cohen bracket
as a natural bracket arising ultimately
from the Leibniz rule.

\subsection{Equivariant algebra}

For other automorphic forms,
analogs of the equivariant
map
of (\ref{CMZ-map}) prove challenging
to identify.
We establish an elementary
algebraic result
to find such equivariant mappings
in terms of the graded
module $\Gr M$ associated to any filtered module $M$.
We use a subscript $i$ to
indicate the $i$-th filtered
or graded component
and work over a field $\mathbb{F}$.

\vspace{1ex}

\noindent
\discolorlinks{ \bf 
\cref{cor:algebraic-Gsplitting}}
{\em
Suppose $M$ is a $\ZZ$-filtered
$\mathbb{F}[G][x]$-module
for a group $G$
with filtration complete
over $\mathbb{F}[x]$.
Say $x$ acts by a scalar on each graded component
of $\Gr(M)$ with mutually distinct
eigenvalues.
For each $i$ in $\ZZ$, the short
exact sequence of $\ \mathbb{F}[G][x]$-modules splits
via a unique $G$-equivariant map
$\tree$:
\begin{equation*}
\begin{tikzcd}[column sep=6ex]
0\arrow{r}
& M_{i-1}
\arrow{r}
& \ \ M_i
\arrow{r}{}
&
(\Gr M)_i
\arrow{r}
\arrow[dotted, bend right]{l}[swap]{\mbox{$\tree$}}
&0\ .
\end{tikzcd}
\end{equation*}
}

\subsection{Jacobi group and pseudodifferential operators}
We apply this poposition to the case
when $G$
is the Jacobi group,
$x$ is a Casimir operator,
and $M$ is a space of pseudodifferential operators
on the Jacobi upper half plane
$\HJ=\{(\tau,z): \tau \in \HS,
z\in \CC\}$.
In contrast to previous work,
we take operators in 
$\del_z$
instead of $\del_\tau$
to directly incorporate the geometry of $\HJ$
instead of relying on 
properties of modular forms;
see \cref{JacobiGroup},
\cref{DefPseudos}, and
\cref{CasimirDefn}.
Here, $M_i$ is the space
of operators with highest
power $i$ on $\del_z$.
These  carry slash actions
that are not merely parametrized by a weight~$k$
and a Jacobi index~$m$, but 
by a weight $k$ and a decomposition
of $m$ as $\ml+\mr$ for left and right Jacobi indices~$\ml$ and $\mr$ in $\CC$,
see \cref{SlashAction}.
We obtain a $G$-equivariant map
(see \cref{SplittingExists})
\begin{gather}
\label{isointro}
\tree: \ 
\CHJ
\xrightarrow{\ \ G\text{-equivariant }\ \ }\ 
\PDO(\del_{z}, C^{\infty}(\HJ))\, 
\end{gather}
that plays the role of 
the map in~\eqref{CMZ-map} in the theory of Cohen-Manin-Zagier
using
\cref{cor:algebraic-Gsplitting}
and a Casimir operator 
(see \cref{Intro:CasimirSubsection}
below
and \cref{Casimir-On-Holos})
with distinct eigenvalues
on the filtered components
$M_i$
for weights $k \geq 2$.
Here, $\tree$ splits the projection 
map 
$
\PDO(\del_{z}, C^{\infty}(\HJ))\, 
\rightarrow
\CHJ
$
onto the top coefficient.

\subsection{Jacobi forms}

We define 
the space
$\rmJ\Psi_{\mlr}$
of {\em Jacobi pseudodifferential operators}
by taking the
pseudodifferential operators
 invariant under
the slash action of the Jacobi group $\Gamma$ 
with a natural growth condition
reflecting that for
the space of Jacobi forms 
$J_{k,m}$,
see \cref{JacobiPseudodifferentialOperators}.
We take invariants
on both sides of~\eqref{isointro}
and impose growth conditions
to obtain a splitting map for Jacobi forms,
see \cref{InducesLinearSplitting}:
$$\tree:
\ \rmJ_{k,m}\ \ 
\ 
\xrightarrow{\ \hphantom{xxxxxx}\ \ }\ 
\Psi\rmJ_\mlr 
\, .
$$
We combine these splitting maps
for $k\geq 2$ to decompose 
the space of pseudodifferential
Jacobi forms:

\vspace{2ex}


\noindent
\discolorlinks{\bf  \cref{cor:isomorphism-to-pseudo-differential-jacobi-forms}}
{\em 
There is a $\CC$-linear isomorphism
from the formal direct product over
$k$ of all 
Jacobi forms 
$J_{k,m}$ of 
weight
$k\geq 2$
and
index $m>0$ 
to the space of Jacobi pseudodifferential operators $\rmJ\Psi_{\mlr}$
with $m=\ml+\mr$
with highest power $-2$ on $\del/\del z$:
\begin{gather*}
\prod_{k\geq 2}  \rmJ_{k,m}
\ \ \cong \ \ 
  \Big( \Psi\rmJ_\mlr 
  \Big)_{-2}
  \, .
\end{gather*}
}

\vspace{2ex}

This corollary then gives
a $1$-complex-parameter
family of Rankin-Cohen brackets
for Jacobi forms 
in the sense of
B\"ocherer~\cite{Boecherer}
as the pull-back 
of the noncommutative multiplication
of pseudodifferential operators
under the 
map $\tree$
(see \cref{Rankin-CohenBracketConstruction}
and \cref{BracketOfJacobiForms}):

\vspace{2ex}

\noindent
\discolorlinks{\bf 
Corollary.} 
{\em
For any parameter $\ixt\in \CC$
and $\nu \geq 0$,
there is a $\CC$-linear Rankin-Cohen bracket
$$
\begin{aligned}
[\ \ \ \ , \ \ \ ]^{\ixt}_{(\nu)}
:\ \ 
\rmJ_{k,m}
\ot\, \rmJ_{k',m'}
\ \longrightarrow\ 
\rmJ_{k+k'+\nu,\, m+m'} 
\, 
\end{aligned}
$$
for any weights $k,k'\geq 2$
and indices $m,m'> 0$
leveraging the noncommutative
multiplication of pseudodifferential
operators.
}

\vspace{1ex}

See \cref{ExplicitBrackets}
for the first few brackets
given explicitly in terms
of partials $\del/\del z$
and $\del/\del \tau$.

\subsection{Casimir operators}
\label{Intro:CasimirSubsection}
A key step in Cohen, Manin, and Zagier~\cite{CMZ}
uses direct calculation to
construct the equivariant 
map~(\ref{CMZ-map})
explicitly.
The analogous calculation
appears more complicated for  other automorphic forms.
Cohen, Manin, and Zagier explain
an alternate approach
(see~\cite[Section 2]{CMZ})
suggested by Beilinson
using a {\em Casimir element}
and the fact that the relevant Lie group is {\em reductive}.
This idea trades 
direct computation of the equivariant map~(\ref{CMZ-map}) 
for the potentially easier computation
of a Casimir operator.
Two problems arise
in trying to deploy this idea
for other automorphic forms:
\begin{itemize}
\item[1.]
The relevant Lie group 
may be {\em non-reductive}. 
Existence
of an analogous equivariant map 
must follow from other facts.
\item[2.]
We may ultimately trade one calculation for another
just as complicated: explicit determination of
the action of the Casimir operator.
(In~\cite{CMZ}, the Casimir operator had a simple
formula.)
\end{itemize}
We show how to exploit 
Beilinson's idea 
even when
the relevant Lie group is not reductive
and/or direct computation of the Casimir operator 
bears little enlightenment.
We construct a unique equivariant map like in~(\ref{CMZ-map})
 for functions on the Jacobi
upper half space $\HJ$
by constructing
a Casimir operator directly
as a product of {\em raising and lowering operators}.
Conley and Raum~\cite{ConleyRaum}
and also Bringmann, Conley, and Richter~\cite{BringmannConleyRichter}
use
Helgason's theories in differential geometry (see~\cite{Helgason1959}) 
to analyze operators covariant
under the action of a  nonreductive 
Lie algebra.  We adapt this approach 
to the setting of infinite-dimensional
vector bundles.
Note that we construct 
the Casimir operator
as composition of operators of degree 1
instead of using 
Helgason's theory~\cite[page 267,
Theorem~10]{Helgason1959}
for universal enveloping algebras.

\subsection{Why Jacobi forms?}
We concentrate on the setting of Jacobi forms
and functions $\HS\times \CC \rightarrow \CC$ 
of a modular variable
$\tau$ in $\HS$ and elliptic variable $z$ in $\CC$ 
with focus on $z$
for several reasons.

First, geometric constructions suggest considering
pseudodifferential operators
in $\del_z$ instead of $\del_\tau$.
A decomposition of the universal enveloping algebra
for the corresponding Jacobi Lie group 
(see~\cite{ConleyRaum}) gives
\begin{gather*}
\mathcal{U}(\mathfrak{g})
\otimes_{\CC[Z]}
\CC(Z)
\ \ \supset\ \ 
\mathcal{U}(\mathfrak{sl}_2(\RR))\otimes_\CC 
\mathcal{U}(\mathfrak{h}_2)\, ,
\end{gather*}
where $\mathfrak{h}_2$ is the complexified Heisenberg Lie algebra
with center generated by $Z$.
We view this as giving
two different kinds of 
pseudodifferential operators:
the first in $\del_ \tau$
arising from $\mathcal{U}(\mathfrak{sl}_2(\RR))$ and
the second in $\del_z$ from $\mathcal{U}(\mathfrak{h}_2)$.
The first kind 
in $\del_\tau$
may be related to modular forms
via the heat operator,
see~\cite{Choie,ChoieEholzer}.
Historically, phenomena attached
to $\mathcal{U}(\mathfrak{h}_2)$ 
suggest
that the second kind 
in $\del_z$
yields interesting
applications.
Indeed, 
modular completions of indefinite theta
functions may be viewed through the lens of 
representations
of $\mathcal{U}(\mathfrak{h}_2)$ 
(see~\cite{BringmannRaumRichter, Raum2015}).
Indefinite theta functions
and their restrictions to modular forms
capture combinatorial data in their coefficients; for example, they
are used to determine combinatorial
congruences, asymptotics, and identities,
see Ono~\cite{OnoUnearthing}.
They also lie at the foundation
of modern understandings of
mock theta series,
see~\cite{ramanujan-1988,Zwegers}.

Second,
our focus on the copy of $\mathcal{U}(\mathfrak{h}_2)$
requires a new 
action on pseudodifferential operators 
(see \cref{cor:action-on-pseudo-differential-operators}), 
in contrast to a mere extension of the action
in Cohen, Manin, and Zagier~\cite{CMZ} (see Lee~\cite{MinHoLee}).

Third,  Jacobi forms 
often help establish
new arguments for
other automorphic forms.
For example, see~\cite{BlochOkounkov},~\cite{ChoieParkZagier},~\cite{Dijkgraaf},~\cite{Ittersum},~\cite{KanekoZagier95},
and~\cite{Zagier16}.

Thus by exploring
pseudodifferential
operators in the 
elliptic variable $\del_z$
instead of the modular variable
$\del_{\tau}$,
we aim to establish new techniques
that go beyond
extending tools for elliptic 
modular forms.
We hope this may open doors to
new 
avenues of research.


\subsection{Outline}

In \cref{JacobiPlane},
we recall 
the action of the Jacobi group $\Gamma$ 
and (extended) real Jacobi Lie group $G$
on the
Jacobi upper half plane $\HJ:=\HS\times \CC$.
We introduce pseudodifferential operators
$\po$ in the elliptic variable $z$ in $\CC$
in \cref{DefnPseudoOperators}
and define an action of the Jacobi group in
\cref{SlashActionSection}
preserving the noncommutative multiplication
of pseudodifferential operators,
see \cref{almostautos}.
A closed form for this action appears
in \cref{cor:action-on-pseudo-differential-operators}.

We turn to filtered modules 
in \cref{sec:algebraic-splitting} 
and give a general algebraic 
condition for $G$-equivariant
splittings in terms of distinct eigenvalues,
see \cref{cor:algebraic-Gsplitting}.
In \cref{sec:casimir-operator},
we construct a Casimir
operator on the space of
pseudodifferential operators
as a product of covariant 
raising and lowering operators
(see~\cref{def:raisingandlowering})
with closed form given in 
\cref{Casimir-Like-On-Smooth-Functions}
and \cref{Casimir-On-Holos}.
We show that this Casimir operator is equivariant 
(see \cref{CasimirEquivariant})
with respect to the slash action.
Note that in Appendix A, 
we explain how these new raising,
lowering, and Casimir
 operators
were found 
through an analysis of
operators 
covariant under 
the action of the
complex Lie algebra $\mathfrak{g}$ 
attached to the Jacobi group $G$.

In \cref{sec:equivariant-splitting-exists},
we use the Casimir operator to
show existence 
of 
an equivariant mapping
as in \eqref{isointro}
(see \cref{SplittingExists} and \cref{TreeRestrictsToInvariants})
and give an explicit closed form.
We give applications in \cref{sec:JacobiForms}
to Jacobi forms.
We define Jacobi pseudodifferential
operators and give
an isomorphism 
to the product space of 
Jacobi forms
in \cref{GammaInvariants}.
This gives rise to new families of
Rankin-Cohen brackets
for Jacobi forms
in \cref{sec:RankinCohen}
using the pull-back
of the noncommutative
multiplication on
pseudodifferential operators.
We also include an iterative approach
for computing Rankin-Cohen brackets.


\subsection{Notation}
For a group $G$ and field $\mathbb{F}$,
we say an $\mathbb{F}$-vector space $M$ is a $G$-module when $M$ is an $\mathbb{F}[G]$-module, and we say any function 
$f: M\rightarrow M'$ of $G$-modules
is $G$-equivariant when $f$
is an $\mathbb{F}[G]$-module homomorphism.
All slash actions are right group actions.
We assume all rings are unital,
all algebras are 
associative $\CC$-algebras,
and tensor products
are taken over $\CC$
unless otherwise indicated:
$\otimes = \otimes_{\CC}$.

\section{Jacobi plane and groups}\label{JacobiPlane}

In this section, we recall the setting of Jacobi forms;
see the
reference texts by Eichler and Zagier~\cite{EichlerZagier} and 
Berndt and Schmidt~\cite{BerndtSchmidt}.

\subsection{The Jacobi upper half plane}
The {\em Jacobi upper half plane}
(of cogenus $1$) is the space
$
  \HJ:=\HS \times \CC,
$
for 
$\HS := \{\tau\in \CC: \Im(\tau) > 0\}$
the usual
Poincar\'e upper half plane.
For $(\tau,z)$ in $\HJ$,
we call 
$$
\begin{aligned}
&\tau \in \mathbb{H}\
\text{the {\em modular variable}, and}\\ 
&z\in \CC\ \text{the {\em elliptic variable}}.
\end{aligned}
$$
We write $C^{\infty}(\HJ)$ 
for the smooth functions, i.e., 
 $\phi: \HJ \rightarrow \CC$
which are smooth as functions
of four real variables.

\subsection{The Jacobi groups
$\Gamma$
and $G$}
\label{JacobiGroup}
Recall the associated Jacobi groups:
$$
\begin{aligned}
&\text{the {\em (full) Jacobi group}}\ \text{ is }\quad
 && \Gamma &:=\ \ 
 & \SL{2}(\ZZ) \ltimes \ZZ^2\, ;
 && \\
&\text{the {\em (centrally extended) real Jacobi group}}\
 \text{ is }\quad
 && G &:=\ \
 & (\text{SL}_2(\RR)\ltimes \RR^2) \times \RR\, 
\text{.} &&
\end{aligned}
$$
Here, the semidirect product is defined via the 
right action of $\SL{2}$ on $\ZZ^2$ or $\RR^2$
(written as row vectors),
$$(M,X)(M',X')=(MM', XM'+X')
\qquad\text{ for } M,M'
\in\text{SL}_2
\quad\text{ and }
X,X'\ \text{in }\ZZ^2 \text{ or } \RR^2\, ,
$$
and  $G$
arises from a nontrivial central extension 
by the additive group $\RR
\cong \{(I,0,\kappa): \kappa\in \RR\}$,
$$
0\longrightarrow
\RR \longhookrightarrow G
\longtwoheadrightarrow \text{SL}_2(\RR)\ltimes \RR^2
\longrightarrow 0,
$$
giving the multiplication 
$$
(M,X,\kappa)\, (M',X',\kappa')
\ :=\ 
\big(MM',\ XM'+X',\ \kappa+\kappa'+ 
\det
\left(\begin{smallmatrix}XM'\\X'\end{smallmatrix}\right)
\,  \big)\ .
$$
In fact, 
$
G\cong \text{SL}_2(\RR)\ltimes H_3(\RR),
$
where $H_3(\RR)$ is the real Heisenberg Lie group of upper triangular
$3\times 3$ matrices.
We identify $\Gamma$ with the
quotient group of~$(\SL{2}(\ZZ) \ltimes \ZZ^2 ) \times \ZZ \subset G$ by its
central subgroup 
$\ZZ\cong \{(I,0,\kappa): \kappa\in \ZZ\}$
throughout.
This implies the group actions of $G$ defined below
(see \cref{GroupActions} and \cref{action})
factor through~$\Gamma$ if they are trivial on~$\ZZ \subset G$.

\subsection{Jacobi cocycles}
\label{JacobiCocycles}
%
Fix a {\em Jacobi index} $m\in \ZZ$.
Let $\beta$
be the usual scalar cocycle with respect
to the left action of $\text{SL}_2(\RR)$ on $\HS$
and let $\alpha_{m}$ be the usual {\em Jacobi cocycle}:
$$ 
\begin{aligned}
\beta: &\ G\times \HJ\rightarrow \CC,
\qquad
&& \, \beta
\big(g, (\tau, z)\big)
= 
(c\tau+d)^{-1}
, \quad\text{ and }
\\
\alpha_m: &\ G \times \HJ\rightarrow \CC,
\qquad
&& \alpha_{m}\big(g, (\tau,z)\big)
= 
\exp\Big(m\, 2 \pi i\, \big( 
\ \sfrac{-c (z + \lambda \tau + \mu)^2}
{(c \tau + d)} + \lambda^2\tau + 2 \lambda z + \lambda\mu+\kappa\big)\Big)\ 
\end{aligned}
$$
for
$g= \big( \left[\begin{smallmatrix}
a & b \\ c & d
\end{smallmatrix}\right]
, (\lambda, \mu),\kappa \big)$
in $G$.
For the function 
defined by holding a group element $g$
fixed, 
we simply write 
$$
\beta(g): \ \HJ \rightarrow \CC
\qquad\text{and}\qquad
\alpha_{m}(g):\ \HJ \rightarrow \CC
\, .
$$
More generally,
for $k$ in $\ZZ$, we define the 
{\em weight $k$, index $m$ cocycle}
\begin{gather}
\alpha_{k,m}:G\times \HJ\rightarrow \CC
\quad\tx{ by }\quad
\alpha_{k,m}:=\beta^k\cdot \alpha_m
\, .
\end{gather}
Then $\alpha_{k,m}: G \rightarrow C^{\infty}(\HJ)$ 
gives a {\em factor of automorphy}.

\subsection{Slash group actions}
\label{GroupActions}

The real Jacobi group $G$ acts on
$\HJ$ by
\begin{gather*}
g(\tau,z) 
 =
  \Big(\, \frac{a \tau + b}{c \tau + d}\, ,\ \frac{z + \lambda \tau + \mu}{c \tau + d} \Big)
\ \in \HJ\ 
\end{gather*}
for any
$g=\big(\left[\begin{smallmatrix}a & b\\c & d \end{smallmatrix}\right],
(\lambda,\mu),\kappa\big)\in G$.
(Note that $\kappa$ plays no role in this action.)  
The action of $G$ on $\HJ$
extends to a family of actions on complex-valued
functions 
parameterized by 
both a {\em weight} $k \in \ZZ$ 
and a {\em (Jacobi) index} $m \in \CC$:
The {\em slash operator $|_{k,m}$ of weight $k$ and index $m$} 
maps a function $\phi :\, \HJ \rightarrow \CC$ 
and element 
$g$ in $G$
to the function
$\big( \phi\, \big|_{k, m}\, g \big): \HJ \rightarrow \CC$
defined by 
$$
\phi\ \big|_{k, m}\ g
\
:=
\
\alpha_{k, m}(g)\
  (\phi\circ g),
$$
i.e.,
$$
\big(\phi\ \big|_{k, m}\ g\big)
(\tau, z)
\ :=\ 
\beta^k(g)\ \alpha_{m}\big(g,(\tau,z)\big)\
  \phi\big(g (\tau, z) \big)
\quad\qquad\text{ for all }
(\tau,z)\in\HJ,\ \ g\in G\, .
$$


\section{Pseudodifferential operators
in the elliptic variable}
\label{DefnPseudoOperators}

We consider pseudodifferential operators
on the Jacobi upper half space $\HJ$.
Note that Cohen, Manin, and Zagier~\cite{CMZ} consider pseudodifferential
operators in the modular variable $\del_{\tau}$, while we
consider pseudodifferential operators in the 
elliptic variable $\del_z$.
Results on modular forms
previously
have been adapted to larger sets of automorphic forms
by focusing on the modular variable $\tau$,
but we turn attention
to the elliptic variable $z$
instead to
capture structures
instrinsic to Jacobi forms
versus
elliptic modular forms.
Coefficients for our pseudodifferential operators are either smooth or holomorphic, i.e., taken from either
$$
C^{\infty}(\HJ)=\{
\phi:
\HJ\rightarrow \CC\  
\text{smooth as a function
of $4$ real variables}\}
\ \ \text{ or }\ \ 
\cO(\HJ):=
\{ \phi:\HJ\rightarrow \CC\
\text{ holomorphic} \}\, . $$
\subsection{Pseudodifferential operators}
\label{DefPseudos}
Let $R=C^{\infty}(\HJ)$ or $\cO(\HJ)$.  We write
$\ddel$ for the 
formal differential operator 
with respect to the elliptic variable $z$ of $\HJ$.
A {\em pseudodifferential operator} with respect to 
$\ddel:=\delz$
with coefficients in $R$ is a formal series 
$$
  \sum_{-\infty< n \ll \infty} 
  \phi_n \ddel^n \quad\text{ with each }
\phi_n \in R\ .
$$
Here, $n\ll \infty $ indicates that
only finitely many terms with positive $n$ are nonzero.
The space of all pseudodifferential
operators is then
\begin{gather*}
\PDO(R, \ddel)\
:=\ \Big\{\  \sum_{-\infty<n \ll \infty} \phi_n \ddel^n:
\phi_n\in R \, \Big\}.
\end{gather*}

\subsection{Noncommutative multiplication}
\label{ssec:pseudodifferential_product}
The algebra of (ordinary) differential operators 
(in the elliptic variable $z$),
$$
\rmD\rmO(R,\ddel)\ :=\
  \Big\{\ \sum_{0\leq n\ll \infty} \phi_n \ddel^n: \ \phi_n\in R\ 
\Big\},
$$
is a subspace of $\PDO(R,\ddel)$
with multiplication
defined by letting $\ddel=\delz$ act as the partial
derivative $\del/\del z$,
\begin{gather*}
\ddel\, \phi = \phi\, \ddel + \tfrac{\del}{\del z}(\phi)
\quad\text{ for all }\phi\in R\, ,
\end{gather*} 
so that
$$
\Big(\sum_{0\leq n \ll\infty}
\phi_n\, \ddel^n \Big)
\Big(\sum_{0\leq m \ll\infty}
\phi'_m \, \ddel^m \Big)
= \sum_{n,m}\
\sum_{0 \le r\leq n} \
\tbinom{n}{r}\ \phi_n\ \tfrac{\del^r}{\del z^r}(\phi'_m)\ \ddel^{n + m - r}
\, .
$$

\vspace{1ex}

For fixed $r$, the binomial coefficent $\binom{n}{r}$
defines a polynomial in $n\geq 0$ that can be extended
to all $n\in \ZZ$.
Hence we may define
a multiplication on
$\PDO(R, \ddel)$ by extending to negative powers 
on $\ddel$,
\begin{gather}
\label{eq:multiplication-of-pseudo-differential-operators}
  \Big( \sum_{-\infty< n \ll \infty} 
  \phi_n\, \ddel^n \Big)
  \Big( \sum_{-\infty< m \ll \infty} 
  \phi'_m \, \ddel^m \Big)
=
  \sum_{n, m}\ \sum_{0 \le r} \
\tbinom{n}{r}\ \phi_n\ \tfrac{\del^r}{\del z^r}(\phi'_m)\ \ddel^{n + m - r}
\, ,
\end{gather}
which equips this space with a non-commutative (associative) algebra structure
(see~\cite{CMZ})
with the space of ordinary differential operators as a subalgebra.

\subsection{Filtration}
\label{Filtration}
We define the {\em order} (filtered degree) of a nonzero pseudodifferential operator
to be the largest $n$ for which the coefficient of 
$\ddel^n$ is nonzero, and we call that nonzero coefficient the 
{\em symbol} or {\em leading coefficient}
of the operator.
Order then defines a natural ascending
$\ZZ$-filtration,
$$\PDO(R, \ddel)= \bigcup_k \PDO(R, \ddel)_{k}\
\qquad\text{ with }\ \ 
\PDO(R, \ddel)_{k}\subset
\PDO(R, \ddel)_{k+1}
\quad\text{ for all $k$ in $\ZZ$}
$$
for $\PDO(R, \ddel)_{k}$
the set of pseudodifferential operators of order $k$ or less,
\begin{gather*}
\PDO(R, \ddel)_{k}=
  \Big\{ 
    \sum_{n\leq k} \phi_n \ddel^{n}
    :\ \phi_n\in R
  \Big\}\, ,
\end{gather*}
since 
$
\PDO(R, \ddel)_{k}\cdot
\PDO(R, \ddel)_{k'}\subset
\PDO(R, \ddel)_{k+k'} 
$
for any integers $k,k'$.
When restricting to elements of
$\PDO(R, \ddel)_{k}$,
we call the coefficient of $\delz^{k}$
the {\em top coefficient} (which may be zero) to distinguish
from the leading coefficient (which is always nonzero).

\section{Slash action on pseudodifferential
operators}
\label{SlashActionSection}

We introduce a right action of the extended
real Jacobi group $G$
on the space of pseudodifferential
operators 
depending on a pair of (left, right) parameters $\ml, \mr$ in $\CC$ which decompose
an index $m$ as $m=\ml+\mr$
and a weight $k$ in $\ZZ$.
We define this action by separating a cocycle
tracking the action into two parts,
one acting on the left and one acting on the right.
We use the parameter $k$
corresponding to weight in order to later consider 
operators 
which change the weight.

\subsection{Slash Action}
\label{SlashAction}
Every pair of Jacobi cocycles
$
  \alpha,\alpha':\,  G
\rightarrow  C^{\infty}(\HJ)
$
as in \cref{JacobiCocycles}
defines a right action of 
$G$ on 
the pseudodifferential operators
$\po$ defined by
\begin{gather*}
  \phi\, \ddel^n\ \  \big|_{(\alpha, \alpha')}\ \ g 
\ =\
  \alpha(g)\  (\phi\circ g)\
\boldsymbol{\del}_{\!z\circ g}^n \ \alpha'(g)
\quad\quad\text{ for }
\phi \in C^{\infty}(\HJ), \ g \in G
\, ,
\end{gather*}
where $z\circ g$ is shorthand notation for the
$z$-coordinate of $g(z,\tau)$:
$$\boldsymbol{\del}_{\! z\circ g}= 
\beta^{-1}(g)\ \ddel =(c \tau + d)\, \ddel
\quad\quad\text{ for }
g=\Big(\big(\begin{smallmatrix}a & b\\c & d\end{smallmatrix}\big),
(\lambda,\mu),\kappa\Big)\in G, \ (\tau,z) \in \HJ\ .
$$
We next write $|_{k, \mlr}$ for the (right) slash action
of $G$ arising from the pair of cocycles 
$\alpha_{k,\ml}$
$\alpha_{\mr}$.

\vspace{0ex}

\begin{definition}
\label{action}{\em
For each $k\in \ZZ$ and 
(left-right pair) $\ml,\mr\in\CC$, 
define a slash action
of $G$ on $\po$,
$$\big|_{k, \mlr}: \ \
\po \times G\ \longrightarrow\ \po
,  $$
by
\begin{gather*}
\begin{aligned}
\phi \, \ddel^n\,
\ \ \big|_{k,\mlr}\ \ g
\ \ :=&\ \
\alpha_{k-n,\ml}(g)
\  (\phi \circ g)\ \ddel^n 
\ \alpha_{\mr}(g)\\
=&\
  \beta^{k-n}(g)\, \alpha_{\ml}(g)\,  (\phi\circ g)\ 
\ddel^n \ \alpha_{\mr}(g)
\quad\quad\text{ for }
\phi \in C^{\infty}(\HJ), \ g \in G.
\end{aligned}
\end{gather*}
}
\end{definition}

\vspace{1ex}

Note that we could instead incorporate the additional factor 
$\beta^{k-n}(g)$ into the second cocycle
without changing the action, as $c \tau + d$ and $\ddel$ commute.
By restricting $C^{\infty}(\HJ)$ to $\holos$,
we obtain a slash action of the real Jacobi group $G$ on $\poh$:
For each $k\in \ZZ$ and $\ml,\mr\in \CC$,
$$\big|_{k, \mlr}: \ \
\poh \times G\ \longrightarrow\ \poh
\, .  
$$
\subsection{Group action on products}
The group $G$ does not act as algebra automorphisms
on $\po$ via this slash operator. 
But if we consider the family
of actions
over all $\ml,\mr\in\CC$,
we uncover a useful transformation law which may be checked directly:
\begin{proposition}\label{almostautos}
The family of slash actions of 
$G$ preserves multiplication
of pseudodifferential operators:
For any $\psi,\psi'$ in $\po$
and $g\in G$:
$$
\big(\psi\ \big|_{k,(m_1,m_2)}\ g\big)
\cdot\big(\psi'\ \big|_{k,(-m_2,m_3)}\ g\big)
\ =\ 
(\psi\cdot\psi')\ \big|_{k,(m_1, m_3)}
\ g 
\qquad\tx{
for all $k\in \ZZ$ and $m_1,m_2,m_3\in\CC$.}
$$
\end{proposition}

\subsection{Slash action in canonical form}
In order to express the slash action 
on pseudodifferential
operators explicitly, we require a little lemma using
the multinomial coefficient
\begin{gather*}
  \sbinom{n}{m_1, \ldots, m_{\ell}}
\ :=\ \ 
  \sfrac{\Gamma(n+1)}
       {\prod_i\,  \Gamma(m_i+1) \rule{0ex}{2ex}}
\  \text{.}
\end{gather*}
Note that this agrees with the combinatorial
definition
$
 \sfrac{n!}
       {\prod_i\, m_i !\,\rule{0ex}{1.5ex}}
       $
for $n, m_i>0$.
For $m_1, \ldots, m_{\ell-1},p \geq 0$ and 
$n\in \ZZ$,
\begin{gather}
\label{multinomial}
 \sbinom{n}{m_1, \ldots, m_{\ell-1}, n-p}
\ =\
 \sfrac{n (n-1) \cdots (n-p+1) }
{\prod_{i} \, m_i !
\rule{0ex}{2ex}}
\, .
\end{gather}

We relegate the proof of the next
 lemma to
 \cref{AppendixB}.
\begin{lemma}\label{triplesum}
In the noncommutative ring $\, \po$,
for all $c_1,c_2 \in \CC$ and $n \in \ZZ$,
\begin{gather*}
  (\ddel + c_1 z + c_2)^n
\ =\ \
  \sum_{0\leq p}\
  \sum_{\substack{ 0\, \le\, t, r, s \\ t+s+2r=p\rule{0ex}{1.5ex}}}
    \sbinom{n}{t, s, 2r, n-p}
\  \sfrac{2^r \, \Gamma(r + \tfrac{1}{2})\rule[-1ex]{0ex}{3ex}}
{\Gamma(\tfrac{1}{2})\rule[0ex]{0ex}{2.5ex}}
\ \  c_1^r\, c_2^s\ (c_1 z)^t\ \ddel^{n - p}
\text{.}
\end{gather*}
\end{lemma}


The preceding lemma 
allows us to give an explicit
formula expressing the slash action of $G$
on pseudodifferential operators
$\po$
with all $\ddel$ to the far right.
Compare with~\cite[Equation~(1.7)]{CMZ}.
Note that setting $\lambda=\mu=\kappa=0$,
$z=0$,
and $\mr=\ml=0$ in the
corollary below gives the action 
of the modular group.
We interpret $0^0$ as $1$ in the next formula.
\begin{cor}
\label{cor:action-on-pseudo-differential-operators}
For any 
$g= \big( \left[\begin{smallmatrix}a & b \\ c & d
\end{smallmatrix}\right], (\lambda, \mu),\kappa \big)$ 
in $G$ and $\ml, \mr\in \CC$ with $m=\mr+\ml$,
$$
\begin{aligned}
&\phi \, \ddel^n\ \ \big|_{k,\mlr}\ \ 
g\ 
=\ 
&\sum_{0\leq p}\sum_{\substack{0\, \le t, s, r\\ t+s+2r=p\rule{0ex}{1.5ex}} }
  \hspace{-.5ex}
  c_g(t,s,r,n)
  \ \sfrac{\alpha_{m}(g, (\tau,z))}{(c \tau + d)^{r +t +s+k-n}}\
  z^t\ (\phi \circ g) 
  \ \ddel^{n - p}
\ \ 
\end{aligned}
$$
for $ n\in \ZZ$ and
$\phi=\phi(z)$ in $C^{\infty}(\HJ)$, where
\begin{gather*}
c_g(t,s,r,n)\ 
:=\ 
\sbinom{n}{t, s, 2r, n-p}\
  \sfrac{2^r\ \Gamma\big(r + \tfrac{1}{2}\big)}
{\Gamma\big(\tfrac{1}{2}\big)\rule{0ex}{2.5ex}}\
  \big(4 \pi i \mr\big)^{r + t + s}\ \
  (-c)^{r+t} \ \big(d \lambda-c\mu\big)^s
\ \  \
\in \CC\, .
\end{gather*}
\end{cor}
\begin{proof}
To express 
$\phi\, \ddel^n\, |\, g
=\beta^{k-n}(g)\, \alpha_{\ml}(g)\, (\phi\circ g)\, 
\ddel^n \alpha_{\mr}(g)$ with $\ddel$ 
to the far right, we 
determine the commutator of $\ddel^n$ and the 
Jacobi cocycle $\alpha_{\mr}:=\alpha_{\mr}\big(g, (\tau,z)\big)$.
Note that
\begin{gather*}
\tfrac{\del}{\del z}\alpha_{\mr}
= \alpha_{\mr}\cdot (c_1z + c_2)
\end{gather*}
for some $c_1,c_2\in \CC$ depending
on $g$, $\tau$, and $\mr$; in fact,
\begin{gather*}
c_1=4\pi i\mr \ \tfrac{-c}{c\tau+d}\, 
\quad\text{and}\quad
c_2=4\pi i \mr\ \tfrac{d \lambda - c\mu}{c\tau+d}\ \, .
\end{gather*}
Thus in the algebra $\po$
(with the noncommutative multiplication
of \cref{eq:multiplication-of-pseudo-differential-operators}),
\begin{gather*}
\ddel\cdot \alpha_{\mr} 
= \alpha_{\mr}\cdot\ddel + \tfrac{\del}{\del z}(\alpha_{\mr})
=\alpha_{\mr}\cdot(\ddel + c_1 z + c_2)
\end{gather*}
and
$$
\ddel^n\cdot \alpha_{\mr}
= \alpha_{\mr}\cdot (\ddel+c_1z+c_2)^n\ .
$$ 
\cref{triplesum}
gives the advertised formula
using the fact that $(c\tau + d)$ and $\ddel$ commute.
\end{proof}

Recall the filtration 
$\po = \bigcup_{k} \po_{k}$ (see \cref{Filtration})
with subscript $k$ indicating the order, i.e., highest power $k$ on $\ddel$ that may appear
with nonzero coefficient.
    
\begin{cor}\label{GroupActionLifts}
Fix $k\in \ZZ$ and $\ml, \mr\in\CC$ 
with $m=\ml+\mr$.
For any $\phi$ in $C^{\infty}(\HJ)$,
$$
\phi\,  \ddel^{-k}\ \big|_{0,\mlr} \ g
\  = \ 
\big( \phi\ \big|_{k,m}\ g\big) \ \ddel^{-k}
\qquad{\rm modulo}\quad
\po_{-k-1}
\, .
$$
    
\end{cor}

\section{An equivariant splitting of filtered modules}
\label{sec:algebraic-splitting}

Throughout this section, we fix a field $\mathbb{F}$
of characteristic zero
and establish a general observation
on filtered modules whose
graded versions decompose into eigenspaces
with distinct eigenvalues for some operator.
We take $R$ to be a commutative ring with unity
in this section.

\subsection{Filtered and associated
graded modules}
Recall that a module $M$ over the (ungraded) ring $R$
is $\ZZ$-filtered when
$M=\cup_{i\in \ZZ} M_i$ with each $M_i$ an $R$-submodule of $M$ with $M_i\subset M_{i+1}$ 
for all $i$:
$$
\ldots \ \subset\ M_{i-2}\ \subset\ M_{i-1}\ \subset\ M_i\ \subset\ M_{i+1}\ \subset\ \ldots 
$$
The associated graded $R$-module
is $$
\Gr M = 
\bigoplus_{i\in \ZZ} \ (\Gr M)_i
\qquad\tx{ for }\qquad 
(\Gr M)_i = M_i/M_{i-1},
\quad i\in \ZZ
\, 
$$
with $R$-action given by $r(m+M_{i-1})=rm+M_{i-1}$
for $r$ in $R$ and $m$ in $M$.  
We write $\pi_i: M_i\rightarrow (\Gr M)_i$
for the canonical projection onto the $i$-th component of the graded module.

\subsection{Complete Filtrations}
\label{CompleteFiltrations}
For a $\ZZ$-filtered $R$-module $M=\cup_{i} M_i$, we consider the usual projective limit,
${\rm proj\ lim}_{i}\ M/M_i
\subset \prod_i (m+M_i)$,
with transition maps 
$M/M_j\rightarrow M/M_i$, 
$m+M_j\mapsto m+M_i$, for $M_j\subset M_i$.
We say the filtration is {\em complete}
over $R$
if the map
\begin{equation}\label{CompleteMap}
M \longrightarrow 
{\rm proj\ lim }_i\ 
(M/M_i)
\, ,
\quad
m\longmapsto \prod_i \ (m +M_i)
\end{equation}
is an $R$-module isomorphism
(see~\cite{AtiyahMacdonald} and~\cite{Weibel}).
We use the next example in 
the proof of \cref{SplittingExists}.

\vspace{2ex}

\begin{example}\label{ExampleCompleteFiltration}{\em 
Consider Laurent series
in the variable $t^{-1}$ for $t$ a formal symbol:
$$
M=R\llp t^{-1} \rrp
=
\Big\{ \ \sum_{-\infty < n \ll \infty} r_n t^{n} : r_n\in R \ \Big\}.
$$
The natural $\ZZ$-filtration is complete  over $R$ 
with $i$-th filtered component the $R$-module
$$
M_i
=
\Big\{ \ \sum_{-\infty < n \le i} r_n t^{n}:r_n\in R \ \Big\}
\qquad\quad\tx{for $i\in \ZZ$}
$$
(so $M_i\subset M_{i+1}$ for all $i$).
Indeed,
$M=\bigcup_i M_i$,
and the product of projection maps
$$
f:\ M\ \longrightarrow\ \prod_i M/M_i,
\quad m\ \longmapsto\ \prod_i m+M_i
\qquad\tx{
given by }\quad
\sum_{n\leq k}r _n t^{n}
\ \longmapsto\  
\prod_i\ 
\Big(\sum_{n\leq k} r_nt^n
+ M_i \Big)
\, ,$$
defines an $R$-module homomorphism 
to the projective limit,
$f: M \rightarrow {\rm proj \ lim}_i \ M/M_{i}$,
which
is injective as $\bigcap_i M_i=0$.
In fact, if $M$ is a module over another ring $S$ 
preserving the filtration, i.e., with
each $M_i$ an $S$-module, then each $M/M_i$ is an $S$-module
and $f$ is also an $S$-module homomorphism, $$f(s m)=\prod_i (sm+ M_i)=s\prod_i (m+M_i)
=s f(m)\, .
$$
Hence the filtration on $M$ is complete
over $S$ as well.
Note that for any fixed $N$ in $\ZZ$,
the induced filtration on $M_N=\cup_{i\leq N}M_i$ 
is also complete
(one may set $(M_N)_i=M_i$ for $i\leq N$
and $(M_N)_i=M$ for $i>N$).
} 
\end{example}

\vspace{2ex}

In the next two propositions, we use
the term ``eigenvalues" loosely
and include scalars
by which $x$ may act on spaces
even when those spaces are zero.

\begin{proposition}
\label{prop:algebraic-splitting}
Let $M=\bigcup_i M_i$ 
be a $\ZZ$-filtered 
$\mathbb{F}[x]$\nbd module 
with filtration complete.
Suppose $x$ acts by a scalar
$\lambda_i$ in $\mathbb{F}$
on each component $(\Gr M)_i$ of the associated graded module with the
eigenvalues $\lambda_i$ for  $i \in \ZZ$
mutually distinct.
Then for every $i$, there is a splitting~$\sigma_i$ to~$\pi_i$,
i.e., an $\mathbb{F}[x]$-module homomorphism
$\sigma_i:(\Gr M)_i\rightarrow M_i$
with $\pi_i\circ \sigma_i =\operatorname{Id}$:
\begin{equation*}
\begin{tikzcd}[column sep=7ex]
0\arrow{r}
& M_{i-1}
\arrow{r}
& \ \ \ M_i
\arrow{r}[swap]{\pi_i}
&
(\Gr M)_i
\arrow{r}
\arrow[dotted, bend right]{l}[swap]{\mbox{$\sigma_i$}}
&0\ .
\end{tikzcd}
\end{equation*}

\end{proposition}
\begin{proof}
Fix some integer $\ell$ 
and without loss of generality assume $\lambda_{\ell}=0$
(else replace $x$ by $x - \lambda_\ell$).
For each $j\neq \ell$,
the composition of
the map
$(\text{Id}- \lambda_j^{-1} x):M\rightarrow M $
with projection onto $M/M_{j-1}$
defines an $\mathbb{F}[x]$-module homomorphism
whose kernel
contains $M_j$ since $x$ acts on $M_j/M_{j-1}$
by the eigenvalue $\lambda_j$
(nonzero as the eigenvalues of $x$ are distinct):
$$ m \longmapsto 
m-\lambda_j^{-1}xm+M_{j-1}
=m-\lambda_j^{-1}\lambda_j m+M_{j-1}
=0
\qquad\text{ for }\quad m\in M_j\, .
$$
We obtain an $\mathbb{F}[x]$-module homomorphism
$$
f_{j}: M/M_j\longrightarrow M/M_{j-1},
\quad
m+M_j \longmapsto (\text{Id} - \lambda_j^{-1}x)m + M_{j-1}
\qquad\text{ for each }
j\neq  \ell\, .
$$
For each $z\in M_{\ell}/M_{\ell-1}$, 
define 
$z_j=0$ for $j \geq \ell$,
$z_{\ell-1}=z$, and 
$$
z_j= f_{j+1}  \cdots f_{\ell-1}(z) \in M/M_{j}
\qquad\text{ for } j< \ell-1
\, .
$$
We claim $\prod_{j} z_j$ 
in $\prod_{j} M/M_{j}$ 
lies in the 
projective limit $\varprojlim_j M/M_j$.
Indeed,
$$x z_{j}= x f_{j+1}\cdots f_{\ell-1}z
=f_{j+1}\cdots f_{\ell-1}xz 
=f_{j+1}\cdots f_{\ell-1} \lambda_{\ell}z =0
\qquad\text{ 
for $j<\ell-1$}
$$
as each $f_i$ is $\mathbb{F}[x]$-linear, so  
if $z_{j}=m+M_{j}$ in $M/M_j$
for $m$ in $M$ and $j<\ell-1$, then
$$ z_{j-1}=f_{j}(z_{j})
= (\text{Id} - \lambda_{j}^{-1}x)m+M_{j-1}
$$
is sent under the transition morphism
$M/M_{j-1}\rightarrow M/M_{j}$
defining the projective limit to 
$$(\text{Id} - \lambda_j^{-1}x)m+M_{j}
= z_j - \lambda_j^{-1} x z_j
= z_{j} - 0
\,  
$$
with $z_{\ell-1}$ sent to $z_{\ell}=0$.

We define $\sigma_{\ell}(z)$ as the preimage
of
$\prod_{j} z_j$ 
under the $\mathbb{F}[x]$-isomorphism
$ M\  \xrightarrow{\ \cong\ }\  \varprojlim_j M/M_j$
of \cref{CompleteMap},
$ m \mapsto \prod_j m+M_j$.
Then $0=z_\ell=\sigma_{\ell}(z)+M_{\ell}$
so $\sigma_{\ell}(z)$ lies in $M_{\ell}$
and
$\pi_\ell \sigma_{\ell}(z)=
\sigma_{\ell}(z)+M_{\ell-1}
=z_{\ell-1}=z$.
Thus $\sigma_{\ell}: M_{\ell}/M_{\ell-1}
\rightarrow M$ is an $\mathbb{F}[x]$-module homomorphism (as it is given by the composition
of $\mathbb{F}[x]$-module maps)
with $\pi_\ell \, \sigma_{\ell}$ the identity.
\end{proof}

In the next corollary,
we consider the group algebra $\mathbb{F}[G]$
of a group $G$ and take
$\mathbb{F}[G][x]$ to be ungraded
(i.e., graded with every element
of degree $0$).

\begin{prop}
\label{cor:algebraic-Gsplitting}
Suppose $M$ is a $\ZZ$-filtered
$\mathbb{F}[G][x]$-module
for a group $G$
with filtration complete
over $\mathbb{F}[x]$.
Say $x$ acts by a scalar $\lambda_i$
in $\mathbb{F}$ 
on each $(\Gr M)_i$ with the
eigenvalues $\lambda_i$ for  $i \in \ZZ$
mutually distinct.
Then for every $i$, there is a unique $\mathbb{F}[G][x]$-module
splitting $\sigma_i$ to~$\pi_i$,
\begin{equation*}
\begin{tikzcd}[column sep=7ex]
0\arrow{r}
& M_{i-1}
\arrow{r}
& \ \ M_i
\arrow{r}[swap]{\pi_i}
&
(\Gr M)_i
\arrow{r}
\arrow[dotted, bend right]{l}[swap]{\mbox{$\sigma_i$}}
&0\ .
\end{tikzcd}
\end{equation*}
In particular, $\sigma_i$ is $G$-equivariant,
i.e.,
for all $g$ in $G$, the following diagram is commutative:
\begin{gather*}
\begin{aligned}
&M_i& &\xleftarrow{\quad\ \sigma_i\quad\ }& &\Gr(M)_i& \\
& g \Big\downarrow&
&&
&\quad
\Big\downarrow g & \\
&M_i& &\xleftarrow{\quad\ \sigma_i\quad\ }& &\Gr(M)_i& 
\, .
\end{aligned}
\end{gather*}
\end{prop}
\begin{proof}
 Note that the action of $G$ and the action of $x$ both preserve
the filtration with actions commuting:
$xg=gx$ as maps on each $M_i$
for all $g$ in $G$.
Fix $i$ and take the $\mathbb{F}[x]$-module splitting
$\sigma_i:(\Gr M)_i\rightarrow M_i$
of \cref{prop:algebraic-splitting}.
To show that~$\sigma_i$ is $G$-equivariant,
we note that the actions of $x$ and
$g$ in $G$ commute with 
the projection map 
$\pi_i: M_i\rightarrow 
(\Gr M)_i$ with $\pi_i \circ \sigma_i=\text{Id}$, 
giving a commutative diagram
for $g$ in $G$
\begin{gather*}
\begin{aligned}
&M_i& &\xleftarrow{\quad\ \sigma_i\quad\ }& &\Gr(M)_i& \\
& g \Big\downarrow&
&&
&\quad
\Big\downarrow g & \\
&M_i& &\xrightarrow{\quad\ \pi_i
\quad\ }& &\Gr(M)_i\ . & 
\end{aligned}
\end{gather*}
Thus for any $z$ in $(\Gr M)_i$,
the element
$(g\sigma_i- \sigma_i g)(z)$
in $M_i$ projects to zero in $(\Gr M)_i$
under $\pi_i$.
Suppose $(g\sigma_i - \sigma_i g)(z)$
were nonzero.  
Then it would lie in 
the filtered component $M_{\ell}$ of $M$ 
for some $\ell<i$  minimal
as the filtration is complete
(so $\bigcap_i M_i=0$).
Then $\pi_\ell(g\sigma_i - \sigma_i g)(z)$
would be nonzero in $(\Gr M)_\ell$.
But since the actions of
$x$ and $g$ commute
and the splitting $\sigma_i$
is an $\mathbb{F}[x]$-map,
this would imply that
$$
\begin{aligned}
x\ (g\sigma_i - \sigma_i g)(z)
=(g\sigma_i - \sigma_i g)(x\, z)
=\lambda_i\, (g\sigma_i - \sigma_i g)(z)
\end{aligned}
$$
and hence
$$
\lambda_i\ \pi_\ell (g\sigma_i - \sigma_i g )(z) 
=
\pi_\ell\left( x\ (g\sigma_i - \sigma_i g)\right)(z)
=
x\ \pi_\ell(g\sigma_i - \sigma_i g)(z)
=
\lambda_\ell\
\pi_\ell(g\sigma_i - \sigma_i g)(z)
$$
which is impossible 
as $\lambda_\ell\neq \lambda_{i}$.
Thus $(g\sigma_i - \sigma_i g)(z)=0$ and
each $\sigma_i$ is $G$-equivariant.

A similar argument verifies that $\sigma_i$ is unique: If
$\sigma_i'$ is another $\mathbb{F}[G][x]$-splitting,
then 
$\pi_i(\sigma_i-\sigma_i')(z)=0$
for nonzero $z$ in $(\Gr M)_i$,
and $(\sigma_i-\sigma_{i}')(z)$ lies in $M_{\ell}$
for some $\ell<i$ with $\ell$ minimal. 
Then 
$x( \sigma_i-\sigma_i')(z)$ is just 
$(\sigma_i-\sigma_i')(xz)
=\lambda_i (\sigma_i-\sigma_i')(z)
$
and thus
$( \sigma_i-\sigma_i')(z)$ must vanish
as
\begin{equation*}
\lambda_i\ \pi_\ell(\sigma_i-\sigma_i')(z)
=
\pi_\ell\ \big(x(\sigma_i-\sigma_i')\big)(z)
=
x \ \pi_\ell(\sigma_i-\sigma_i')(z)
=\lambda_\ell\ \pi_\ell(\sigma_i-\sigma_i')(z)
\, .
\qedhere
\end{equation*}
\end{proof}

\vspace{2ex}

\begin{remark}{\em 
\label{BoundedFiltrations}
\cref{prop:algebraic-splitting}
and
\cref{cor:algebraic-Gsplitting}
both extend to modules with filtration
bounded above.  Indeed,
if the eigenvalues are distinct
for $i<N$, then there is a splitting
map for $i<N$ as in the conclusion
of the two propositions.
}
\end{remark}


\section{Raising and lowering operators
and a Casimir operator}
\label{sec:casimir-operator}

We will construct a linear map
$$\tree:\holos\longrightarrow \poh$$
explicitly in the next section
using the action of a Casimir operator on $\poh$
defined in this section.
We first introduce raising and lowering operators on $\po$
and then define our Casimir operator as a composition of these
operators.
In~\cref{FindingCasimir},
we use analysis of the complexified Lie algebra
$\frak{g}$ of $G$ to explain
our construction of these operators
and of the Casimir operator $\cC_{k,\mlr}$ on $\po$.

\subsection{A Casimir on $\HJ$}
\label{BoringCasimir}
Berndt and Schmidt~\cite{BerndtSchmidt} 
identified 
a Casimir operator on $\HJ$
(see also
Pitale~\cite{Pitale}).
We consider the {Casimir operator} $\mathcal{C}_{k,m}: C^{\infty}(\HJ)
\rightarrow C^{\infty}(\HJ)$
constructed by
Conley and Raum~\cite[Theorem~2.4]{ConleyRaum}
for $k\in \ZZ$ and $0\neq m\in \CC$ 
(see \cref{Casimir-Top-Piece})
and restrict to holomorphic functions:
$$
\text{\em the Casimir operator}\quad
\mathcal{C}_{k,m}:\ \ \holos\rightarrow \holos
\quad\text{ 
acts by the scalar $\quad$
$2\pi i m(k^2-3k)$.}
$$
We also use the standard
{heat operator}
from~\cite[p.~32]{EichlerZagier}:
$$
\text{\em the heat operator}\quad
\quad 
\bbL_{m}
:\ \  \holos\rightarrow \holos
\qquad\text{ is defined as }
\quad
\bbL_{m}=8\pi i m \tfrac{\del}{\del \tau}
- \tfrac{\del^2}{\del z^2}\, .
$$
We extend
$\bbL_{m}$ to $m=0$.
See~\cite{AthreyaLagaceMollerRaum}
for a Casimir operator when $m=0$.

\subsection{Raising and lowering operators}

Berndt and Schmidt~\cite{BerndtSchmidt}
use techniques from reductive groups
to analyze the Jacobi group $G$,
although it is not reductive,
defining raising and lowering operators
on $C^{\infty}(\HJ)$.
We introduce related raising and lowering operators on
the space of pseudodifferential forms.
We write $\phi_\tau$, $\phi_z$, $\phi_{\ov \tau}$, $\phi_{\ov z}$
for the partial derivative of $\phi$ in $\CHJ$
with respect to $\tau, z, \ov \tau, \ov z$, respectively,
where for $\tau=x+iy$ and $z=u+iv$,
\[
\begin{aligned}
&
\tfrac{\del}{\del \tau}
=
\tfrac{1}{2}\big(\tfrac{\del}{\del x} - i \tfrac{\del}{\del y}\big),
&
\quad
&
\tfrac{\del}{\del \ov\tau}
=
\tfrac{1}{2}\big(\tfrac{\del}{\del x} + i \tfrac{\del}{\del y}\big),
&
\quad
&
\tfrac{\del}{\del z}
=
\tfrac{1}{2}\big(\tfrac{\del}{\del u} - i \tfrac{\del}{\del v}\big),
&
\quad
&
\tfrac{\del}{\del \ov z}
=
\tfrac{1}{2}\big(\tfrac{\del}{\del u} + i \tfrac{\del}{\del v}\big)\ . \\
&
\end{aligned}
\]

We justify the following definition
of covariant differential operators
in \cref{FindingCasimir}
(see \cref{AltBasisIsRaisingLowering}).
Here,
the superscript of $J$ indicates
origin with respect to 
the elliptic variable $z$
as opposed to the modular variable $\tau$.
We scale the operators to match 
generators of the corresponding Lie algebra $\mathfrak{g}$ from
\cite{ConleyRaum}.
\begin{definition}
\label{def:raisingandlowering}
{\em
Fix $k\in \ZZ$ and $\ml,\mr \in \CC$.
Define
\begin{gather*}
\begin{aligned}
\rule{0ex}{2.5ex}
&\text{{\em raising operators}}&
&\Rmod_{k,\mlr},  &&\Rell_{k,\mlr}:
&&\po\rightarrow\po, \\
\rule{0ex}{2.5ex}
&\text{{\em lowering operators}}&
&\Lmod_{k,\mlr}, &&\Lell_{k,\mlr}:
&&\po\rightarrow\po,
\quad\text{and}\\ 
\rule{0ex}{2.5ex}
&\text{{\em constant operators}}&
&\Cmod_{k,\mlr}, &&\Cell_{k,\mlr}:
&&\po\rightarrow\po
\end{aligned}
\end{gather*}
by
$$
\begin{aligned}
\Rmod
(\phi\, \ddel^n) 
&=
  \big(2 i 
  (
  \phi_\tau
  + \tfrac{v}{y}\, \phi_z
  + 2 i \pi m \tfrac{v^2}{y^2} \phi
  )
  + \tfrac{k - n}{y} \phi
  \big)\, \ddel^n 
  + 4 \pi i m_r \tfrac{n v}{y^2} \phi\, \ddel^{n-1}
  + \pi m_r \tfrac{n (n - 1)}{y^2} \phi\, \ddel^{n-2}
\text{,}
\\[.3em]
\Lmod
(\phi\, \ddel^n )\ 
&=
  - 2 i y \big(
  y \, \phi_{\ov \tau}
  + v \, \phi_{\ov z}
  \big)\, \ddel^n
  - \pi m_r n (n - 1) \phi\, \ddel^{n-2}
\text{,}
\\[.3em]
\Rell
(\phi\, \ddel^n)\ 
&=
  \big(i
  \phi_z
  - 4 \pi  m \tfrac{v}{y} \phi
  \big)\, \ddel^n
  + 2 \pi i m_r \tfrac{n}{y} \phi\, \ddel^{n-1}
\text{,}
\\[.3em] 
  \Lell
(\phi\, \ddel^n) \ 
&=
  -i y\, \phi_{\ov z}\ \ddel^n
  + 2 \pi i  m_r n\, \phi\, \ddel^{n-1}
\text{,}
\\[.3em]
\Cmod
(\phi\, \ddel^n)\  
&=
  (k - n) \phi\, \ddel^n
  + 4 \pi i m_r \tfrac{n v}{y} \phi\, \ddel^{n-1}
  + 2 \pi m_r \tfrac{n (n - 1)}{y} \phi\, \ddel^{n-2}\text{,}
\\[.3em]
\Cell
(\phi\, \ddel^n)\  
&=
  2 \pi i m \, \phi\, \ddel^n
  \rule[-3ex]{0ex}{2ex}
\end{aligned}
$$
for $\phi$ in $\CHJ$ and $n$ in $\ZZ$,
where $m=\ml+\mr$,
with subscripts $k, \mlr$ on operators suppressed.
}
\end{definition}

The raising, lowering, and constant operators
are covariant maps:
\begin{prop}
\label{cor:covariantoperators1}
For each $k\in \ZZ$ and $\ml,\mr\in \CC$,
the operators 
$$
\Rmod:=\Rmod_{k,\mlr},\quad
\Lmod:=\Lmod_{k,\mlr},\quad
\Rell:=\Rell_{k,\mlr},\quad
\tx{and }\ \ 
\Lell:=\Lell_{k,\mlr}
$$
are covariant,
and the operators $\Cmod:=\Cmod_{k,\mlr}$
and $\Cell:=\Cell_{k,\mlr}$ are invariant:  
\begin{gather*}
\begin{aligned}
&\Rmod\big(\psi  
\ \big|_{k,\mlr}\ g\big)
&=& \quad
\Rmod(\psi)  
\ \big|_{k+2,\mlr}\ g\, ,&\
&\Lmod\big(\psi  
\ \big|_{k,\mlr}\ g\big)
&=& \quad
\Lmod(\psi)
\ \big|_{k-2,\mlr}\ g\, ,&\\
&\Rell\big(\psi  
\ \big|_{k,\mlr}\ g\big)
&=& \quad
 \Rell(\psi)  
\ \big|_{k+ 1,\mlr}\ g\, ,&\
&\Lell\big(\psi  
\ \big|_{k,\mlr}\ g\big)
&=& \quad
\Lell(\psi)
\ \big|_{k-1,\mlr}\ g\, ,&\\
&\Cmod\big(\psi  
\ \big|_{k,\mlr}\ g\big)
&=& \quad
\Cmod(\psi) 
\ \big|_{k,\mlr}\ g \, ,&\ 
&\Cell\big(\psi  
\ \big|_{k,\mlr}\ g\big)
&=& \quad
\Cell(\psi) 
\ \big|_{k,\mlr}\ g \ & 
\end{aligned}
\end{gather*}
for all $g\in G$ and $\psi\in\po$.
\end{prop}
\begin{proof}
Direct computation verifies the claim
using generators for $G$, 
for example, 
$$
\Big(\big(\begin{smallmatrix}1 & b\\0 & 1\end{smallmatrix}\big),
(0,\mu), 0 \Big),
\quad
\Big(\big(\begin{smallmatrix}0 & -1\\1 & 0\end{smallmatrix}\big),
(0,0), 0 \Big)
\qquad
\tx{for $b, \mu$ in $\RR$.}
$$
Alternatively, one may  
adapt the more theoretical
approach
of~\cite{ConleyRaum}
for finding covariant operators
to the setting of infinite-dimensional
$K$-vector bundles
for $K=\tx{stab}_G(i,0)$
using the fact that $G/K\cong \HJ$ is a reductive
coset space although $G$ is not
reductive,
see Appendix A.
\end{proof}

\vspace{2ex}

\begin{remark}{\em
\label{RaisingLowering}
\cref{cor:covariantoperators1} may be 
rephrased succinctly in terms
of {\em raising or lowering}
the weight $k$
giving the slash group
action $|_{k,\mlr}$ on $\po$
(compare with the operators of~\cite{BerndtSchmidt}):
$$
\begin{aligned}
&\Rmod  &&\ \tx{ raises by } 2,
\qquad
&&\Lmod  &&\ \tx{ lowers by } 2,
\\
&\Rell &&\  \tx{ raises by } 1, 
\qquad
&&\Lell &&\ \tx{ lowers by } 1,
\ \ \tx{and}
\\
&\Cmod  &&\  \tx{ preserves weight}  ,
\qquad
&&\Cell &&\  \tx{ preserves weight} 
.
\end{aligned}
$$
}
\end{remark}
\subsection{A Casimir operator}
We construct the Casimir operator
on $\po$ 
as a composition of degree $1$
operators.  These arise from
a convenient basis for 
the complexified Lie algebra
$\mathfrak{g}$ of $G$ 
and
a Casimir element in the 
center of 
the universal enveloping algebra
$\mathcal{U}(\mathfrak{g})$
identified in~\cite{ConleyRaum};
see Appendix~A.

\vspace{2ex}

\begin{definition}
\label{CasimirDefn}
{\em 
For $k\in \ZZ$ and $\ml,\mr\in \CC$ with~$m = \ml + \mr \ne 0$,
define a {\em Casimir operator}
($\CC$-linear map)
$$\mathcal{C}_{k,\mlr}:\ \ \po\longrightarrow\po$$
as this composition 
(applied left to right) of raising, lowering, and constant covariant
operators:
\begin{gather*}
\mathcal{C}_{k,\mlr}=
4\Cell\,
{\Rmod}\,\Lmod
+ 2i\, (\Rell)^2\,\Lmod
-2i\, \Rmod\,(\Lell)^2
+2i\, \Rell\,\Lell\,
(\Cmod-2) - 3 \Cell\,\Cmod + \Cell\,{\Cmod}^2
\, .
\end{gather*}
}
\end{definition}

\vspace{2ex}

\cref{cor:covariantoperators1}
implies that 
the Casimir operator on $\po$
is equivariant with respect
to the slash action:

\begin{theorem}
\label{CasimirEquivariant}
Fix $k\in \ZZ$ and $\ml,\mr\in \CC$ with~$m = \ml + \mr \ne 0$.
The Casimir operator
$\mathcal{C}_{k,\mlr}$ is $G$-equivariant
with respect to the 
slash action $\ |_{k, \mlr}$  on $\po$:
\begin{gather*}
\mathcal{C}_{k,\mlr}
\big( \psi \ \big|_{k, \mlr}\ \ g\big)
\ =\
\mathcal{C}_{k,\mlr}( \psi )
\ \ \big|_{k,\mlr}\ \ g
\end{gather*}
for all $g$ in $G$ and $\psi$ in $\po$.
\end{theorem}

\subsection{Casimir operator explicitly}
We verify that the action of the Casimir operator
 on $\po$
 restricts to an action
 on $\poh$
 using the following
 formula for
$\cC_{k,\mlr}$
obtained by
straightforward calculation using
\cref{def:raisingandlowering} and 
\cref{CasimirDefn}.

\begin{prop}
\label{Casimir-Like-On-Smooth-Functions}
Fix $k\in \ZZ$ and $\ml,\mr\in \CC$ with~$m = \ml + \mr \ne 0$.
The Casimir operator $\cC_{k, \mlr}$ acting on the space of pseudodifferential operators $\po$ is
\begin{gather*}
\cC_{k, \mlr} \big( \phi \ddel^n \big)
\ \ =\ \  
  \cC_{k, \mlr}^{(0)} (\phi)\, \ddel^n
  +
  \cC_{k, \mlr}^{(1)} (\phi) \, \ddel^{n-1}
  +
  \cC_{k, \mlr}^{(2)} (\phi) \, \ddel^{n-2}
\text{,}
\end{gather*}
where
\begin{alignat}{2}
\nonumber
\cC_{k, \mlr}^{(0)} (\phi)
&\ ={}&&
  2 \pi i m\, \big( (k - n)^2 - 3 (k - n) \big)\, \phi
  +
  32 \pi i m\, y^2\, \phi_{\tau\ov{\tau}}
  -
  8 \pi m\, (1 - 2k + 2n) y \, \phi_{\ov{\tau}}
\\ \nonumber
&&
  +\,&
  32 \pi i m\, y v\, \phi_{\tau \ov{z}}
  -
  4 y^2\, \big( \phi_{\tau\ov{z}\ov{z}} + \phi_{\ov{\tau} zz} \big)
  -8 \pi m\, (1 - k + n) v \, \phi_{\ov{z}}
  -
  8 \pi i m\, v^2\, \phi_{\ov{z}\ov{z}}
\\ \nonumber
&&
  +\,&
  2 i\, (k - n) y\, \big( \phi_{z\ov{z}} + \phi_{\ov{z}\ov{z}} \big)
  -4 y v\, \big( \phi_{z\ov{z}\ov{z}} + \phi_{zz\ov{z}} \big)
\, \text{,}
\\ \nonumber
  \cC_{k, \mlr}^{(1)} (\phi)
&\ ={}&&
  4 \pi i m_r\, n (1 - k + n)\, \big( \phi_z + \phi_{\ov{z}} \big)
  +16 \pi m_r\, n y\, \big( \phi_{\tau \ov{z}} - \phi_{\ov{\tau} z} \big)
  -  8 \pi m_r\, n v\, \big( \phi_{z \ov{z}} + \phi_{\ov{z}\ov{z}} \big)
\, \text{,}
\\ \nonumber
  \cC_{k, \mlr}^{(2)} (\phi)
&\ ={}&&
  16 \pi^2 \ml \mr \, n (n - 1)\, \big( \phi_{\tau} + \phi_{\ov{\tau}} \big)
  + 2 \pi i m_r\, n (n - 1) \,
  \big( \phi_{zz} + 2 \phi_{z \ov{z}} + \phi_{\ov{z}\ov{z}} \big)
\, \text{.}
\end{alignat}
\end{prop}

\vspace{2ex}

\begin{remark}{\em 
\label{Casimir-Top-Piece}
Observe that the zero-graded piece 
$\cC_{k, \mlr}^{(0)}$ 
as an operator on $C^{\infty}(\HJ)$
depending on $n$
coincides with the Casimir operator
$\cC_{k-n, m}$ acting on 
$C^\infty(\HJ)$
(with slash action $|_{k - n,m}$)
given in~\cite{ConleyRaum}.  
}\end{remark}

\vspace{2ex}

We would expect the Casimir operator $\cC_{k,\mlr}$ to 
take the space of holomorphic 
pseudodifferential operators to itself if $G$ were reductive.
As $G$ is non-reductive, we instead use
\cref{Casimir-Like-On-Smooth-Functions}:
\begin{cor}
\label{Casimir-On-Holos}
Fix $k\in \ZZ$ and $\ml,\mr\in \CC$ with~$m = \ml + \mr \ne 0$.
The Casimir operator 
$\cC_{k,\mlr}$ on pseudodifferential operators
$\po$
restricts to a Casimir operator
$$\cC_{k,\mlr}: \ \ 
\poh\longrightarrow \poh$$
given by 
\begin{gather*}
\begin{aligned}
\cC_{k,\mlr}(\phi\, \ddel^n)
\ \ \ =\ \quad
&2 \pi i m\, \big( (k - n)^2 - 3 (k - n) \big)\ \phi\  \ddel^n\\
+ &4 \pi i m_r\, n (1 - k + n)\ \
\tfrac{\del}{\del z}(\phi)\  \ddel^{n-1}\\
- &
  2 \pi i m_r\, n (n - 1) \ \ 
  \bbL_{\ml}(\phi)\  
 \ddel^{n-2}\, 
\end{aligned}
\end{gather*}
for  $\phi\in\holos$.
In particular,
for $k=0$,
\begin{gather*}
\cC_{0,\mlr}
(\phi\, \ddel^n) =
 2 \pi i m\, n(n+3)\, \phi\, \ddel^n
+  4 \pi i m_r\, n (n+1)\, 
\tfrac{\del}{\del z}(\phi)\ \ddel^{n-1}
- 
  2 \pi i m_r\, n (n - 1)\
  \bbL_{\ml}(\phi)\ \ddel^{n-2}.
\end{gather*}
\end{cor}

\vspace{1ex}

\section{Equivariant splitting
for pseudodifferential operators}
\label{sec:equivariant-splitting-exists}

In this section, we identify a family of 
linear maps
from the space of
pseudodifferential operators
to the holomorphic functions on the Jacobi
upper half place,
$$\tree: \poh \longrightarrow \holos\, , $$
equivariant under the action of the 
 real Jacobi group $G$.
If this acting Lie group $G$ were reductive, 
general theory would predict existence of 
a map $\Upsilon$ splitting
projection onto the top coefficient
on each filtered component.
Since $G$ is not
reductive, we use other arguments
to prove 
that a splitting map exists.
We identified a Casimir
operator $\cC_{k, \mlr}$ acting on $\poh$ 
in \cref{sec:casimir-operator}. 
Here, we take advantage of the fact
that the eigenvalues of
the induced action 
on the graded pieces are distinct
and apply
\cref{cor:algebraic-Gsplitting} to the weight~$k$ slash action on 
pseudodifferential operators.

To distinguish various actions of $G$,
we write $(M,\ |_{k,m})$ for 
any $G$-module $M$ defined by a right action
of $G$ by a slash operator $|_{k,m}$,
and we write $(M,\ |_{k,m})^G$ or $M^G$
for its submodule of $G$-invariants.
For a fixed $k$ in $\ZZ$,
we abbreviate $\pi$ for
the map 
projecting each operator 
 in the component of filtered
degree $-k$
to its top coefficient, i.e.,
the coefficient
of $\ddel^{-k}$,
see \cref{Filtration}:
\begin{equation}
\label[projection]{PiMap}
\pi=\pi_{-k}: \ \ 
\poh_{-k}\longrightarrow \holos,
\qquad\
\sum_{n = 0}^\infty\phi_n \ddel^{-n-k}
\longmapsto
\phi_{0}
\, .
\end{equation}

\subsection{Exact sequence}
The natural filtration 
on $\Psi=\poh$ (see \cref{Filtration})
induces an exact sequence of $G$-modules
for each $k$ in $\ZZ$ and 
$\ml,\mr$ in $\CC$ with $m=\ml+\mr$:
\begin{equation}
\label{exactsequence}
\begin{tikzcd}[column sep=6ex]
0\arrow{r}
& \big( \Psi_{-k-1}\ ,\ \big|_{0,\mlr}\ \big)
\arrow{r}
&\big( \Psi_{-k}\ ,\ \big|_{0,\mlr}\ \big)
\arrow{r}{\pi}
&\big(\holos\ ,\ \big|_{k,m}\ \big)
\arrow{r}
\arrow[dotted, bend right]{l}[swap]{\mbox{$\tree$}}
&0\ .
\end{tikzcd}
\end{equation}
We construct an explicit $G$\nbd 
equivariant splitting map
for each $k$
and $m\neq 0$,
$$
\tree=\tree_{k,\mlr}:\ \holos \longrightarrow \poh_{-k}
\, ,
$$
i.e., an equivariant map with $\pi 
\circ \Upsilon = \text{Id}$,
the identity,
so that 
taking $G'$-invariants of each term in the exact sequence
for a suitable subgroup
$G'$ of $G$
(see \cref{sec:JacobiForms})
gives an exact
sequence of invariants 
(recalling that the invariant-functor is generally
only left-exact):
\begin{equation}
\label{eq:splitting-of-invariants-exact-sequence}
\begin{tikzcd}[column sep=6ex]
0\arrow{r}
& \big( \Psi_{-k-1}\ ,\ \big|_{0,\mlr}\ \big)^{G'}
\arrow{r}
&\big( \Psi_{-k}\ ,\ \big|_{0,\mlr}\ \big)^{G'}
\arrow{r}{\pi}
&\big(\holos\ ,\ \big|_{k,m}\ \big)^{G'}
\arrow{r}
&0\ .
\end{tikzcd}
\end{equation}

\vspace{1ex}

\subsection{Recursion on coefficients}
We show the existence of  
the $G$-equivariant splitting map
$\Upsilon$ of exact sequence~\eqref{exactsequence}
in \cref{SplittingExists} below;
we will use the next lemma  
to explicitly describe
this map in that same result.
\begin{lemma}
\label{RecursiveFormula}
Fix $2\leq k\in \ZZ$ and $\ml,\mr\in \CC$ with $m=\ml+\mr \ne 0$.
Suppose a $\CC$-linear function
$
Y:
\cO(\HJ) \rightarrow \poh_{-k}
$
preserves the leading coefficient
and commutes with Casimir actions, i.e.,
for all $\phi\in \cO(\HJ)$,
\begin{gather*}
\begin{aligned}
Y(\phi )\ &=\ \phi\, \ddel^{-k} + \sum_{0< n} \phi_n\, \ddel^{-k-n}
\quad\quad \text{ with } \phi_n\in\cO(\HJ)
\quad\quad \text{ and}
\\
Y\big(\mathcal{C}_{k,m}(\phi)\big)
\ &=\ \mathcal{C}_{0,\mlr}\big(Y(\phi)\big)\, .
\quad\quad 
\end{aligned}
\end{gather*}
Then the coefficients of\, $Y(\phi)$ satisfy 
the recursion $\phi_{-1}=0$,
$\phi_{0}=\phi$, and
\begin{gather*}
  \phi_n\
=\  
  \sfrac{\mr (k+n - 1) (k+n - 2)}{mn(3 - n - 2k)}
  \ \big(
  2 \tfrac{\del}{\del z}(\phi_{n - 1})
  -\bbL_{\ml}(\phi_{n-2}) \big)
\quad\quad\text{for}\quad n > 0
\, .
\end{gather*}
\end{lemma}
\begin{proof}
Recall that
$\mathcal{C}_{k,m}$ acts by a scalar on $\holos$ (see \cref{BoringCasimir}),
$$
2 \pi i m (k^2 - 3k)\, Y(\phi)
\ =\ 
Y\big(  \cC_{k,m}(\phi) \big)
\ = \
\mathcal{C}_{0,\mlr}\big(Y(\phi)\big),
$$
giving an eigenvector equation.
We equate the coefficients of $\ddel^{-k-n}$
using
\cref{Casimir-On-Holos}.
On the left side, this coefficient
is just $2 \pi i m (k^2 - 3k)\phi_n$,
whereas on the right side, this coefficient is
$$
-2\pi i\mr(k+n-2)(k+n-1)\, \bbL_{\ml}(\phi_{n-2})
+4\pi i \mr(k+n-1)(k+n-2)\tfrac{\del}{\del z}\phi_{n-1}
+2\pi im(k+n)(k+n-3)\phi_n
$$
giving the recursive
formula in the lemma.
\end{proof}

\vspace{1ex}

\subsection{
Using the Casimir to find the equivariant splitting map}
We describe the $G$-equivariant splitting map $\tree$ recursively
using
the explicit formula for $\cC_{k, \mlr}$ 
in \cref{sec:casimir-operator}
and the heat operator
$\bbL_{m}=8\pi i m \tfrac{\del}{\del \tau}
- \tfrac{\del^2}{\del z^2}$
before giving a closed formula
in \cref{ClosedFormulaForTree}.
We restrict to weights $k\geq 2$
 in order to
 secure distinct eigenvalues
needed to invoke
 \cref{cor:algebraic-Gsplitting}.

\begin{theorem}
\label{SplittingExists}
For all $2\leq k\in \ZZ$ and $\ml,\mr\in \CC$ with $m=\ml+\mr \ne 0$,
there is a $G$-equivariant
$\CC$-linear
map 
\begin{gather*}
\tree = \tree_{k,\mlr}:\ \ 
\big(\holos,\ \big|_{k,m}\ \big) \longrightarrow 
\big(\poh_{-k},\ \big|_{0,\mlr}\ \big)
\end{gather*}
that commutes with the Casimir operators
and preserves the leading coefficient,
i.e.,
$$
\begin{aligned}
&1)\ \  \tree\big(\phi\ \big|_{k, m}\,g\big) 
\ = \ 
\tree(\phi)\ \big|_{0,\mlr}\ g\, 
& & 
\text{ for all } g \in G,
\\
&2)\ \  \tree\big(\cC_{k,m}(\phi)\big)
= \cC_{0, \mlr}\big(\tree (\phi)\big)\, ,
& & \text{ and}
\\
&3) \ \ \tree(\phi) 
\ = \ 
\phi\, \ddel^{-k} + \sum_{0<n} \phi_n \ddel^{-k-n}
& & \text{ for some } \, \phi_n \in\holos,
\end{aligned}
$$
for all $\phi$ in $\holos$.
Furthermore, $\tree(\phi)$ is given
in terms of the
heat operator:

\begin{gather*}
 \phi_n\ 
=
  \sum_{i\in \ZZ:\ 0\leq i \leq \tfrac{n}{2}}
c_{n, i}\
\Big(\big(\tfrac{\del}{\del z}\big)^{n - 2i}\ \bbL_{\ml}^{\! i}\Big)
(\phi)
\quad\text{ for }n>0
\end{gather*}
for $c_{n,i}$ defined recursively by
$c_{0, 0} = 1$, $c_{*,-1}=0$,
$c_{n,i}=0$ for $i>n/2$, 
and 
$$
  c_{n,i}
=
  \sfrac{\mr (k+n -1) (k+n-2)\rule[-.5ex]{0ex}{2ex}}
{m n (3-n-2k)\rule{0ex}{1.5ex}}
\
  \big(
  2 c_{n-1,\, i}
  \mathbin{\,-\,}
  c_{n-2,\, i-1}
  \big)
  \qquad\text{ for } i\leq n/2
\, .
$$
\end{theorem}

\vspace{2ex}

\begin{remark}
\label{ClosedFormulaForTree}
{\em 
We may expand the equivariant
splitting map $\Upsilon$ as
$$
\begin{aligned}
\tree(\phi)
&=
\phi \, \ddel^{-k}
- \sfrac{\mr k}{m}
\tfrac{\del}{\del z}(\phi) \, \ddel^{-k-1}
+
\sfrac{m_r k(k+1)}
{2 m^2 (1-2k)}
\Big( 
 (m-2 \mr k)
\tfrac{\del^2}{\del z^2}(\phi)
-
(8 \pi i m \ml)
\tfrac{\del}{\del \tau}(\phi) 
\Big)
\, \ddel^{-k-2}
+\ 
\cdots
\, .
\end{aligned}
$$
}
\end{remark}

\vspace{1ex}

\begin{proof}[Proof 
of \cref*{SplittingExists}] 
We regard $M=\poh_{-2}$ as a filtered 
$\holos$-module: $M=\bigcup_{i\leq -2} M_i$
where $M_i = \poh_i$.
The associated graded module is
$\Gr M
=\bigoplus_{i\leq -2} (\Gr M)_i 
= \bigoplus_{i\leq -2} M_i/M_{i-1}$
as in \cref{Filtration}.
We regard $M$ 
as a $\CC[x]$-module 
by defining an action of $x$
on each $M_k$ by
$x z = \cC_{0, \mlr}(z)$
(extended via composition).
Then $\Gr M$ inherits
both the $\CC[x]$-module structure
and the $G$-module structure
on $M$
as both $x$ and $g$ in $G$
preserve the filtration.
The Casimir
operator $\mathcal{C}_{0, \mlr}$ 
is $G$-equivariant by \cref{CasimirEquivariant}
so the action of $G$ and $x$ commute,
and 
hence $M$ is a $\CC[G][x]$-module.

Fix $k\geq 2$ 
and identify
$(\Gr M)_{-k}$
with $\holos$ 
via
$\phi\, \ddel^{-k} +M_{-k-1}\mapsto \phi$.
The induced action of $\CC[x]$
on $\holos$ under this identification
is given by 
$$x \, \phi 
=2 \pi i m (k^2-3k)\, \phi
= \cC_{k, m}(\phi)
\qquad\tx{ for }
\phi \in \holos
$$
by
\cref{Casimir-On-Holos} (see \cref{BoringCasimir}),
so $(\Gr M)_{-k}$ is an eigenspace
for the action of $x$ with eigenvalue
$2\pi i m (k^2-3k)$.
Furthermore, we observe that the induced 
action of $G$ on 
$\holos$ under this identification
(via the
slash operator $\big|_{0,\mlr}$ on $M$)
is given by
the slash operator $\big|_{k,m}$
by \cref{GroupActionLifts}.

The filtration on $M$
is complete over $\CC[x]$,
see \cref{CompleteFiltrations}
and \cref{ExampleCompleteFiltration}
with $R=\holos$, $t=\ddel$, 
and $S=\CC[x]$.
Then as the eigenvalues of $x$
on the spaces $(\Gr M)_{i}$ are distinct for distinct $i\leq -2$,
\cref{cor:algebraic-Gsplitting}
(see \cref{BoundedFiltrations})
implies 
the existence of 
a unique 
$\CC[G][x]$-module homomorphism
splitting 
the projection $\pi_{-k}:
\poh\rightarrow \holos$
onto the top coefficient, 
i.e.,
a $G$-equivariant map 
$$\tree=\tree_{k, \mlr}:=\tree_{k,\mlr}: \ \ 
\holos\longrightarrow \poh_{-k}
\qquad\quad
$$
with $\pi_{-k} \tree$
the identity,
for each $k\geq 2$.
The recursion in
\cref{RecursiveFormula} 
gives the advertised form for $\tree$.
\end{proof}

\vspace{2ex}

\begin{remark}\label{CommutativeDiagram}
{\em 
For all $2\leq k\in \ZZ$ and $\ml,\mr\in \CC$ with $0\neq m=\ml+\mr$,
the splitting map,
\begin{gather*}
\Upsilon:
\ \big(\holos\ ,\ \big|_{k,m}\ \big)
\longrightarrow
\big( \poh_{-k}\ ,\ \big|_{0,\mlr}\ \big)
\, ,
\end{gather*}
of \cref{SplittingExists} gives
(see the proof)
the following 
commutative diagrams
(for all left arrows or all right arrows)
for each $g$ in the Jacobi group $G$, 
where $\pi$ again is the projection
onto the coefficient of $\ddel^{-k}$:
\[
\begin{tikzcd}[column sep=9ex,row sep=8ex]
\holos 
\arrow[swap]{d}{ \mathcal{C}_{k, m} }
\arrow[shift right=-.5ex]{r}{\tree}
%
& \poh_{-k}
\arrow{d}{\mathcal{C}_{0,\mlr}} 
\arrow[shift right=-.5ex]{l}{\pi}
\\
\holos 
\arrow[shift right=-.5ex]{r}{\tree}
& \poh_{-k}
\arrow[shift right=-.5ex]{l}{\pi}
\end{tikzcd}
\hspace{5ex}\text{ and }\hspace{5ex}
\begin{tikzcd}[column sep=9ex,row sep=8ex]
\holos 
\arrow[swap]{d}{ g}
\arrow[shift right=-.5ex]{r}{\tree}
& \poh_{-k}
\arrow{d}{g} 
\arrow[shift right=-.5ex]{l}{\pi}
\\
\holos 
\arrow[shift right=-.5ex]{r}{\tree}
& \poh_{-k}
\arrow[shift right=-.5ex]{l}{\pi}
\end{tikzcd}
\]
}
\end{remark}

\vspace{2ex}

\begin{remark}{\em
Setting $\mr=0$ recovers
the traditional landscape of Jacobi forms.
Indeed, when $\mr=0$, 
we set $m=\ml$, and
the action of $G$ on $\poh$
corresponds to an action of $G$
on $\holos$ by the slash operator of index $m$
and weight given by the negative power of $\ddel$:
\begin{gather*}
\phi\, \ddel^{-k}\ \big| \ g
\ =\ \big(\phi\ \big|_{k,m} \ g\big)\ \ddel^{-k}.
\end{gather*}
Thus $\ddel$ serves
merely as a placeholder in this case: its power
indicates the weight of the slash operator
by which $G$ should act on the coefficient.
The recursion in 
\cref{RecursiveFormula}
implies that all the coefficients of $\Upsilon_{k, (\ml,0)}$ 
but one vanish and  
\begin{gather*}
  \tree_{k,(\ml,0)}(\phi)
\ =\ 
  \phi\,\ddel^{-k}
\quad\text{ for all}\quad \phi\in\holos\ .
\end{gather*}
Thus the splitting map in this case is simply the identity
map augmented by notation to indicate a choice of
weight for the action of $G$ on $\holos$.
}\end{remark}

\section{Applications to Jacobi forms}
\label{sec:JacobiForms}

We show in this section that the 
formal product of Jacobi forms 
over all weights
with a fixed index is in bijection with a set of
invariant pseudodifferential operators.
In this section, we will
restrict the
indices~$m$ to integers so that the slash actions of $G$
induce slash actions of
~$\Gamma$
on both
$\holos$ and $\poh$
(since the integral points of the center of $G$
act trivially).

\subsection{Jacobi forms}
A {\em Jacobi form} is 
an analogue of a modular form for the Jacobi upper half plane
instead of the Poincar\'e upper half plane, i.e., 
a complex-valued function 
on the Jacobi upper half plane $\HJ$
with good growth behavior 
which is invariant under the action of the Jacobi group 
$\Gamma =\SL{2}(\ZZ) \ltimes \ZZ^2$
via the slash-operator.

\vspace{2ex}

\begin{definition}
\label{def:jacobiforms}
{\em
A {\em (holomorphic) Jacobi form}
of weight $k \in \ZZ$ and index
$m \in \ZZ$ 
is a holomorphic function 
$\phi :\, \HJ \rightarrow \CC$
\begin{enumerateroman}
\item that satisfies $\phi \big|_{k, m}\, \gamma = \phi$ 
for all $\gamma \in \Gamma$, and
\item has growth given by
$e^{2 \pi i\, m \, \lambda^2 \tau} \phi(\tau, \lambda \tau + \mu) = O(1)$ 
for all $\lambda, \mu \in \QQ$.
\end{enumerateroman}
We denote the space of (holomorphic) Jacobi forms of 
fixed weight $k$ and fixed index $m$ by $\rmJ_{k, m}$.
}
\end{definition}
\vspace{2ex}

Recall that $\rmJ_{k,m}=0$ 
for $k<0$ or $m<0$ and
$\rmJ_{k,m}$ recovers 
the space of elliptic modular
forms for $m=0$.
Under the usual multiplication of complex-valued functions,
the space of Jacobi forms
is a ring bigraded by weight and index:
If $\phi$ and $\phi'$ are Jacobi forms of respective weights $k$ and $k'$
and respective indices $m$ and $m'$, 
then $\phi\,\phi'$ is a Jacobi form of
weight $k+k'$, and index $m+m'$.
We set $\rmJ$ to be the bigraded ring over all weights $k$ 
and all Jacobi indices $m$,
$$
{\rm J}_{\bullet,\bullet}\ :=\ \bigoplus_{k,m} {\rm J}_{k,m}\ .$$

\subsection{Jacobi pseudodifferential operators}
\label{JacobiPseudodifferentialOperators}
We define
Jacobi pseudodifferential
operators as those invariant
under the slash action
with an eye to the growth condition:

\begin{definition}
\label{JacobiPseudodifferentialDef}
{\em
The space of {\em Jacobi pseudodifferential
operators} for $\ml,\mr\in\CC$
with $m=\ml+\mr\in\ZZ$
using the slash action $|_{0, \mlr}$ of $\ \Gamma$ 
is
$$
  \rmJ\Psi_\mlr
:=
  \Big\{\ 
  \sum_{k \ge 0} \phi_k\, \ddel^{-k} \in 
  \Big(\poh\Big)^\Gamma
  \,:\ \ 
  e^{2 \pi i\, m\, \lambda^2 \tau}\,
  \phi_n(\tau, \lambda \tau + \mu) = O(1)\;
  \tx{for all $\lambda, \mu \in \QQ$}
  \Big\}
\tx{.}
$$
}
\end{definition}

\vspace{1ex}

Jacobi pseudodifferential operators inherit a filtration from pseudodifferential operators
with filtered component of degree $-k$ given by
\begin{gather}
\label{PseudoDifferentialJacobiFormsFiltration}
  \Big( \rmJ\Psi_\mlr \Big)_{-k}
\ :=\ 
  \rmJ\Psi_\mlr \cap \poh_{-k}\, 
\tx{.}
\end{gather}
As expected, the map $\pi_{-k}:
\sum_{n\geq 0}
\phi_n\del^{-k-n}
\mapsto \phi_0$ projecting
any pseudodifferential operator
of filtered degree $-k$
to its top coefficient
takes Jacobi pseudodifferential operators to Jacobi forms,
see \cref{CommutativeDiagram}:
$$\pi_{-k}\Big(\, \big(\rmJ\Psi_{\mlr}\big)_{-k}\, \Big) 
\ \subset\ 
J_{k,\mr+\ml}
\quad\text{ for all }
k\in \ZZ, \text{ and }
\ml,\mr \in \CC
\tx{ with } m = \ml+\mr \in \ZZ
\, .
$$

The noncommutative product
of pseudodifferential operators
invariant under the Jacobi group $\Gamma$
is again invariant when we 
adjust the action appropriately,
i.e., when the subscripts recording
the actions align:
For any $\ixt$ in $\CC$,
by \cref{almostautos}
(see 
\cref{JacobiPseudodifferentialDef}),
\begin{equation}
    \label{product}
\psi\in \rmJ\Psi_{(m-\ixt,\, \ixt)}
\qquad\text{ and }\qquad
\psi'\in
\rmJ\Psi_{(-\ixt,\, \ixt+m')}
\qquad\text{ implies }\qquad
\psi\cdot \psi'
\in 
\rmJ\Psi_{(m-\ixt,\, \ixt+m')}
\, .
\end{equation}
Notice here that the growth condition 
in \cref{PseudoDifferentialJacobiFormsFiltration} is respected by the product formula in \cref{ssec:pseudodifferential_product}.
Thus the natural noncommutative multiplication map
on pseudodifferential operators
restricts 
for each fixed $m,m'$ in $\ZZ$ and~$\ixt$ in~$\CC$
to a multiplication map
$$
\rmJ\Psi_{(m-\ixt,\, \ixt)}
\ \ot\ \rmJ\Psi_{(-\ixt,\, \ixt+m')}\
\longrightarrow \
\rmJ\Psi_{(m-\ixt,\, \ixt+m')}
\, .
$$

\subsection{Splitting map for Jacobi pseudodifferential operators}

As a consequence
of the theory developed in the last section,
we obtain a splitting map
of $\Gamma$-invariants
for $m$ in $\ZZ$.
Note that we restrict to 
index $m\neq 0$.
When $m=0$, the top coefficient
$\phi_0$ of a Jacobi pseudodifferential operator 
is a Jacobi form which is
constant in the elliptic
variable $z$,
and we fall into the classical
setting of modular forms.
In addition, 
our splitting map $\tree$
is not defined for $m=0$,
so it does not give a decomposition
of 
$\rmJ\Psi_\mlr$ directly.

\begin{cor}\label{TreeRestrictsToInvariants}
Fix $2\leq k\in \ZZ$ and $\ml,\mr\in \CC$ with $0\neq m=\ml+\mr\in \ZZ$.
The  
$\CC$-linear map
\begin{gather*}
\tree = \tree_{k,\mlr}:\ \ 
\big(\holos,\ \big|_{k,m}\ \big) \longrightarrow 
\big(\poh_{-k},\ \big|_{0,\mlr}\ \big)
\end{gather*}
of \cref{SplittingExists}
restricts to a $\CC$-linear
map of\/ $\Gamma$-equivariants
$$
\begin{tikzcd}[column sep=7ex]
\Big(\,\holos\, \Big)^\Gamma
\arrow{r}{\tree}
&
\Big( \poh_{-k}\ \Big)^{
\Gamma}
\arrow[bend left]{l}[swap]{\pi}
\end{tikzcd}
$$
with $\pi\circ \tree = \operatorname{Id}$,
the identity map,
for $\pi=\pi_{-k}$
and $\Gamma$ acting on the domain
by the slash action $|_{k,m}$ and 
acting on the codomain
by $|_{0, \mlr}$.
\end{cor}

Now we consider the
growth condition.
\begin{cor}
\label{InducesLinearSplitting}
For any 
$2\leq k\in \ZZ$ and $\ml,\mr\in \CC$ with $0\neq m=\ml+\mr\in \ZZ$,
the map $\tree_{k, \mlr}$ 
of \cref{SplittingExists}
restricted to $\rmJ_{k,m}$
induces a $\CC$-linear splitting 
$$
\begin{tikzcd}[column sep=7ex]
 \tree= \tree_{k, \mlr}
:\ \rmJ_{k,m}\ \ 
\arrow{r}{\tree}
&
  \Big( \rmJ\Psi_\mlr \Big)_{-k}
\arrow[bend left]{l}[swap]{\pi}
\,
\end{tikzcd}
$$
to the projection map $\pi=\pi_{-k}$
restricted to 
$\big( \rmJ\Psi_\mlr \big)_{-k}$,
i.e.,
$\pi\circ \tree = 
\operatorname{Id}$
on $J_{k,m}$.
\end{cor}
\begin{proof}
By \cref{SplittingExists},
it suffices to check that 
$\Upsilon_{k,\mlr}(\phi)$ lies in $\rmJ\Psi_{\mlr}$ for $\phi$ in $\rmJ_{k,m}$, for fixed $k$ and $m$.
Holomorphicity and invariance
follow from by \cref{TreeRestrictsToInvariants},
so we
focus on the growth condition.
We use the recursive formula for $\tree_{k, \mlr}$
from
\cref{SplittingExists}
and check that 
$(\tfrac{\del}{\del z})^{n-2i}\ \bbL_m^i\,
\phi$ satisfies the growth condition. 
By the theta-decomposition of Jacobi forms, this amounts to checking the growth condition for
quasi-modular forms, where it is obvious.
\end{proof}

\subsection{Modeling
Jacobi forms as power series}
To formulate applications to Jacobi forms,
we use a formal variable $X$ with exponent
indicating weight of a Jacobi form.
Define
$$
\holos\llb X^{-1}\rrb
= 
\Big\{\ \sum_{k \ge 0} \phi_k \, X^{-k}:\
\phi_k\in \holos
\Big\}
$$
with an slash action of $G$ 
for fixed $m\in \CC$
defined by 
\begin{equation}\label{ProdReps}
\big(\sum_{k} \phi_{k}X^{-k} \big)\ 
\ \big|_{\bullet, m}\  g
\ = \ 
\sum_{k} \ \big(\phi_{k} \ \big|_{k,m}\ g \big)\ X^{-k}
\qquad\text{ for }
\phi_k\in \holos, \ \ 
g\in G
\, .
\end{equation}
We take the natural filtration 
on $M=\holos\llb X^{-1} \rrb$
and set
$$M_{-2}= \Big\{\ \sum_{k\geq 2} \phi_k \, X^{-k}:\ \phi_k\in \holos
\ \Big\},
$$
the natural filtration on $\poh$
as in \cref{Filtration},
and use the splitting maps $\tree_{k, \mlr}$ from
~\cref{SplittingExists}.

\begin{theorem}
\label{EquivariantIsomorphism}
 For any $\ml,\mr\in\CC$ with $m=\ml+\mr\neq 0$,
the map
\begin{align*}
\tree_{\mlr}
: \ \ \ 
\big( \holos\llb X^{-1}\rrb _{-2},\ \big|_{\bullet, m}
\ \ \big)
&\ \longrightarrow  \
\big(\poh_{-2},\ \big|_{0,\mlr}\, \big)\, ,
\\
  \sum_{k \ge 2} \phi_k X^{-k}
\ &\ \longmapsto\ 
  \sum_{k \ge 2} \tree_{k,\mlr}\big( \phi_k \big)
\end{align*}
is a $G$\nbd equivariant filtered $\CC$-linear isomorphism
with slash action $|_{\DOT, m}$ as given in \cref{ProdReps}.
\end{theorem}
\begin{proof}
For $g$ in $G$,
\cref{SplittingExists} implies that $\tree=\tree_{\mlr}$
takes
$(\sum_k \phi_k X^{-k})\ \big|_{\bullet, m}\ g
=\sum_k 
\big(\phi_k\ \big|_{k,m}\ g\big) X^{-k}
$
to
$$
\begin{aligned}
\sum_k\tree_{k, \mlr}
\big( \phi_k\ \big|_{k, m}\ g \big)
&\ =\
\sum_k \Big( \tree_{k, \mlr} (\phi_k) \ \big|_{0,\mlr}\ g \Big)
\\ &\ =\
\Big(\sum_k\tree_{k, \mlr}(\phi_k)\Big)\ \big|_{0,\mlr}\ g 
\\ &\ =\
\tree\big(\sum_k\phi_kX^{-k} \big)\ \big|_{0,\mlr}\ g\, ,
\end{aligned}
$$
where the second equality holds as there
are only finitely many terms in the sum
with nonzero coefficient of a fixed $\ddel^{-n}$
and the action of $g$ respects the filtration.

We construct a $G$-invariant inverse map
$\poh\rightarrow 
\holos\llb X^{-1}\rrb$,
\begin{gather*}
\psi = \sum_{n\geq 2}\phi_n\, \ddel^{-n}
\ \longmapsto\ \displaystyle\sum_{n\geq 2}\, \phi'_{n}\, X^{-n}\, ,
\end{gather*}
by setting
$\phi'_{2}:=\phi_{2}$,
subtracting
$\tree_{2,\mlr}(\phi_2)$
from $\psi$ 
to cancel top coefficients,
and then using the new top coefficient
as the coefficient $\phi_3'$
of $X^{-3}$, etc:
We define recursively
$$\phi_k' := \pi_{-k}\Big(\psi-\sum_{n = 2}^{k-1}\tree_{n,\mlr}(\phi'_{n}) \Big)
\quad\text{ 
for $k \geq 2$,}
$$
where $\pi_{-k}$ again
 is the projection
onto the (top) coefficient of $\del^{-k}$ as in \eqref{PiMap} and
\eqref{exactsequence}.
One may check that this map
is both left and right inverse to
$\tree_{\mlr}$
(by inducting on filtered degree $k$
and working modulo $\poh_{-k}$).
\end{proof}

\begin{cor}\label{GammaInvariants}
 For any $\ml,\mr\in\CC$ with $0\neq m=\ml+\mr\in \ZZ$, the map of \cref{EquivariantIsomorphism}
restricts to a\ $\CC$-vector
space isomorphism of invariants:
\begin{gather*}
  \Upsilon_{\mlr}
:\
  \Big(\holos\llb X^{-1}\rrb_{-2},\ \big|_{\DOT,m}\ \Big)
  ^\Gamma
\ \ \overset{\cong}{\longrightarrow}\ \ 
  \Big( \poh_{-2},\ \big|_{0,\mlr} \ \Big)^\Gamma
\,\tx{.}
\end{gather*}
\end{cor}

\vspace{2ex}

We now observe that this last map
is compatible with the growth conditions
defining Jacobi forms
and Jacobi pseudodifferential operators.
\begin{cor}
\label{cor:isomorphism-to-pseudo-differential-jacobi-forms}
 For any $\ml,\mr\in\CC$ with $0\neq m=\ml+\mr\in \ZZ$,
the map $\tree_{\mlr}$ 
of \cref{EquivariantIsomorphism} 
induces a linear
isomorphism from the formal direct product over
$k$ of all Jacobi forms of index $m$ and weight
$k$ to the space of Jacobi
pseudodifferential operators:
\begin{gather*}
  \tree_{\mlr}
:\,
\prod_{k\geq 2}  \rmJ_{k,m}\ \ 
\overset{\cong}{\longrightarrow}
\ \ 
  \big( \rmJ\Psi_\mlr \big)_{-2}
\,\text{,}
\qquad
\prod_{k\geq 2}
 \phi_k 
\longmapsto
  \tree_\mlr\big( \sum_{k \ge 2} \phi_k X^{-k} \big)
=
  \sum_{k \ge 2} \tree_{k,\mlr}\big( \phi_k \big)
\end{gather*}
\end{cor}
\begin{proof}
We identify
$\prod_{k\geq 2}  \rmJ_{k,m}$
with the subspace of
$(\holos\llb X^{-1}\rrb_{-2})^\Gamma$
  under the slash action $\big|_{\DOT,m}$
satisfying the necessary
growth condition of 
\cref{def:jacobiforms}
in the obvious way, i.e., 
via $\prod_{k\geq 2} \phi_k \mapsto 
\sum_{k\geq 2} \phi_k X^{-k}$.
The map $\tree_{\mlr}$
of \cref{GammaInvariants}
restricts to an injective linear
map taking $\prod_{k\geq 2}  \rmJ_{k,m}$
to
  $\big( \rmJ\Psi_\mlr \big)_{-2}$
 by \cref{InducesLinearSplitting}
and \cref{GammaInvariants}.
One may show that the inverse map 
constructed
in the proof of \cref{EquivariantIsomorphism}
is compatible with the growth
condition and therefore
restricts to a map 
that takes $\big( \rmJ\Psi_\mlr \big)_{-2}$
to $\prod_{k\geq 2}  \rmJ_{k,m}$,
so $\tree_{\mlr}$ restricts
to an isomorphism as advertised.
\end{proof}

Note that the linear isomorphism of
\cref{cor:isomorphism-to-pseudo-differential-jacobi-forms}
is not a ring homomorphism
under multiplication of Jacobi forms
and multiplication of pseudodifferential operators.
In the next section, we define the Rankin-Cohen
brackets as a measure of the failure 
of the map $\tree$ to respect
these multiplications,
or rather,
as the mechanism to transport
the noncommutative multiplication 
of pseudodifferential operators
to Jacobi forms.


\section{Rankin-Cohen brackets for Jacobi forms}
\label{sec:RankinCohen}

We now define 
distinguished families
of Rankin-Cohen
brackets for Jacobi forms
using the correspondence 
with invariant pseudodifferential operators
established in the previous section.
These Rankin-Cohen brackets 
record the noncommutative
multiplication of 
pseudodifferential operators
transported
to the product space
of Jacobi forms.
Indeed,
we define the Rankin-Cohen
bracket as the pull-back
of the noncommutative multiplication
under the map
$\tree$ from Jacobi forms to pseudodifferential operators 
as in \cref{cor:isomorphism-to-pseudo-differential-jacobi-forms}.

\subsection{Rankin-Cohen brackets}
\label{Rankin-CohenBracketConstruction}
We say a $\CC$-bilinear function
$\cO(\HJ)\times \cO(\HJ)
\rightarrow
\cO(\HJ)
$
is a $\nu$-th {\em Rankin-Cohen bracket on Jacobi forms}
(see~\cite{Boecherer})
for fixed $k,k', m,m' \geq 0$
if it raises weight by $\nu$ 
and is a holomorphic 
differential operator,
i.e., has the form
$$
f:\
(\phi, \phi')\
\mapsto
\sum_{a,b, c, d\geq 0}
c_{a,b,c,d}\ \
\tfrac{\del^a}{\del \tau^a}
\tfrac{\del^b}{\del z^b}
(\phi)\
\tfrac{\del^c}{\del \tau^{c}}
\tfrac{\del^d}{\del z^{d}}
(\phi')
$$
with the sum finite,
$c_{a,b,c,d}$ in $\CC$,
and
$$
f(\phi,\phi')\, |_{k+k'+\nu,m+m'} \ g
\ = \ 
f\big(\phi|_{k,m} \ g,\
\phi'\, |_{k',m'}\ g\big)
\qquad\text{ for all }
g \in G
\, .
$$
Every $\nu$-th Rankin-Cohen bracket
restricts to a $\CC$\nbd{}bilinear function
$$
J_{k,m}\, \times\, J_{k',m'}
\longrightarrow J_{k+k'+\nu,\, m+m'}\, .
$$

B\"ocherer~\cite{Boecherer}
showed that the $\CC$-dimension of
the space of $\nu$-th Rankin-Cohen brackets 
for 
fixed weights $k$, $k'$ and indices
$m$, $m'> 0$
is generically
$1+ \lfloor \nu/2 \rfloor$.
Choie and Eholzer~\cite{ChoieEholzer} found a basis for this space 
by leveraging the heat operator, effectively tapping into elliptic modular forms via the theta decomposition,
but
no particularly
privileged basis 
has been identified.
This is in stark
contrast to the case
of Rankin-Cohen brackets
of elliptic modular forms:
there the space of
$\nu$-th Rankin-Cohen
brackets
is $1$-dimensional for $\nu$ even
and $0$ for $\nu$ odd.
Thus for elliptic modular forms,
fixing a choice of Rankin-Cohen bracket as $\nu$ ranges amounts to choosing
some coefficients in $\CC$.

For Jacobi forms, one seeks 
families of Rankin-Cohen
brackets that capture
key properties of the Jacobi setting.
Choie~\cite{Choie} identified a one-dimensional subspace 
of Rankin-Cohen brackets
for each $\nu$
arising as
polynomials in the heat operator
(see~\cite{Boecherer}).
In the next theorem, we define
families of Rankin-Cohen
brackets of Jacobi forms parametrized by a parameter $\ixt$ in $\CC$ 
with geometric origin.
We use the
noncommutative multiplication
of pseudodifferential operators
given
in terms of the maps
$$
  \tree_{\mlr}
:\,
\prod_{k\geq 2}  \rmJ_{k,m}\ \ 
\overset{\cong}{\longrightarrow}
\ \ 
  \big( \rmJ\Psi_\mlr \big)_{-2}
$$
of \cref{cor:isomorphism-to-pseudo-differential-jacobi-forms}.
We convert two Jacobi forms into two
operators using $\tree$, multiply the operators,
and then apply $\tree^{-1}$ to convert back into
a family of Jacobi forms.
The $\nu$-th element
of this family
then will define the $\nu$-th Rankin-Cohen bracket.
In the next theorem, we
embed
each space $J_{k,m}$
of 
Jacobi forms for fixed $m$ and $k\geq 2$
into the product $\prod_{k'\geq 2} J_{k',m}$
as the $k$-th component
with other components zero.
\begin{thm}
\label{BracketOfJacobiForms}
Fix a complex parameter $\ixt$.
For any two Jacobi forms
$\phi$ and $\phi'$ of
respective weights $k, k'\geq 2$ and indices $m, m'>0$, define
a family of Rankin-Cohen brackets
 by
$$
\prod_{i\geq k+k'} 
[\phi, \phi' ]^{\ixt}_{(i-k-k')}
\ \ =\ \ 
\tree^{-1}_{(m-\ixt, \, \ixt + m')}
\Big( 
\tree_{(m-\ixt, \ixt)}(\phi)
\cdot
\tree_{(-\ixt, \ixt+ m')
}(\phi')
\Big)
\, ,
$$
i.e., 
$[\phi, \phi' ]^{\ixt}_{(\nu)}$
is the $(\nu+k+k')$-th coordinate
of the image under $\, \tree^{-1}$
for $\nu\geq 0$.
Then each 
$[\phi, \phi' ]^{\ixt}_{(\nu)}$
is a Jacobi form of weight
$k+k'+\nu$ and index $m+m'$,
and 
$$\prod_{\nu\geq 0} \
[\ \ \ \ , \ \ \ ]^{\ixt}_{(\nu)}
:\ \ \
\rmJ_{k,m}
\ot\ \rmJ_{k',m'}
\ \longrightarrow\ 
\prod_{\nu\geq k+k'}\
\rmJ_{k+k'+\nu, m+m'} 
\, 
$$
defines $\,\CC$-linear function
for $k,k'\geq 2$ and $m,m'>0$. 
\end{thm}
\begin{proof}
We use the multiplication 
\cref{product} on pseudodifferential Jacobi forms.
\cref{InducesLinearSplitting,cor:isomorphism-to-pseudo-differential-jacobi-forms}
imply 
that for any weights $k,k'\geq 2$,
indices $m,m'> 0$, 
and parameter $\ixt\in \ZZ$,
the following
composition 
is well-defined:
$$
\begin{aligned}
\rmJ_{k,m}\
\ \ot \ \
 \rmJ_{k',m'}\ \ 
& \xrightarrow{\ \ 
 \tree_{k, (m-\ixt,\, \ixt)}
\, \ot\,  \tree_{k', (-\ixt,\, \ixt+m')}\ \ }
&& \big( \rmJ\Psi_{(m-\ixt,\, \ixt)} \big)_{-k}
 \ \ot \ 
  \big( \rmJ\Psi_{(-\ixt,\, \ixt + m')} \big)_{-k'}
\\ 
& \xrightarrow{\qquad \text{multiplication}\qquad}
&& 
\big( \rmJ\Psi_{(m - \ixt,\, \ixt+m')} \big)_{-k-k'}
\\
& \xrightarrow{\quad\quad\ \ \
(\tree_{(m-\ixt,\, \ixt+m')})^{-1} \ \ \ \ \ }
&& 
\hspace{-1ex} 
\prod_{\nu\geq 0}\
\rmJ_{k+k'+\nu,\, m+m'} 
\, .
\end{aligned}
$$
\end{proof}



\vspace{2ex}

We next consider
 a subvariety
of the space of all Rankin-Cohen brackets as 
considered by B\"ocherer~\cite{Boecherer}.

\begin{cor}
For each fixed $\nu\geq 0$,
the family of Rankin-Cohen brackets
$[\ \ , \ \ ]^{\ixt}_{(\nu)}$
defined in \cref{BracketOfJacobiForms}
as $\ixt$ ranges
over complex values defines
a subvariety of lines of 
expected dimension $1$ 
in the space of
all 
Rankin-Cohen brackets 
for Jacobi forms 
of fixed respective weights $k$, $k'\geq 2$ and indices
$m$, $m'> 0$.
\end{cor}
\begin{proof}
First note that the map
$\tree_{k, \mlr}$
of \cref{InducesLinearSplitting}
takes each Jacobi form $\phi$
to a pseudodifferential
operator whose
coefficient of each $\ddel^n$
is a finite sum of 
a finite number of 
compositions
of 
partial derivatives $\del/\del\tau$
and $\del/\del z$ 
applied to $\phi$
with coefficients
in $\QQ$,
see \cref{ClosedFormulaForTree}. 
The $\nu$-th Rankin-Cohen bracket
of \cref{BracketOfJacobiForms} is defined
by applying the maps 
$\tree_{k,(m-\ixt, \ixt)}$
and $\tree_{k', (-\ixt, \ixt + m')}$
of 
\cref{InducesLinearSplitting}
to two input Jacobi forms
$\phi$ and $\phi'$ respectively,
multiplying the resulting pseudodifferential
operators, and pulling back via 
the map 
$\tree^{-1}_{m-\ixt, \ixt + m'}$
of \cref{cor:isomorphism-to-pseudo-differential-jacobi-forms}.
We can thus express 
the coefficient
of each $\ddel^n$
in each pseudodifferential operator
as a finite sum
of mixed partial derivatives 
with each coefficient 
in $\QQ[\ixt]$,
and the same can be said
for their product.
Pulling
back via
$\tree^{-1}_{m-\ixt, \ixt + m'}$
gives
a Jacobi form
for each $\nu$
expressed as a finite
sum of mixed
 partial derivatives 
 in $\del/\del\tau$
and $\del/\del z$  
of the original Jacobi
forms 
with coefficients 
in $\QQ[\ixt]$
by the proof of
\cref{EquivariantIsomorphism},
and we obtain the advertised subvariety.
\end{proof}
\vspace{2ex}

\subsection{Iterative approach
to Rankin-Cohen brackets}
The proof of \cref{EquivariantIsomorphism}
gives a recursive construction
for the Rankin-Cohen bracket
of two Jacobi forms 
which we review now.
Suppose Jacobi forms $\phi$, $\phi'$ 
have respective weights $k$, $k'\geq 2$
and indices $m$, $m'> 0$.
The zero-th Rankin-Cohen bracket is
just their product:
$$\left[\phi,\phi'\, \right]^{\ixt}_{(0)}=\phi\cdot \phi'
\qquad\text{ is 
a Jacobi form of weight $k+k'$ and index $m+m'$.}
$$
We observe that the top terms
cancel in $\tree(\phi)\tree(\phi')$
and $\tree(\phi\, \phi')$
and the first Rankin-Cohen bracket
$\left[\phi,\phi'\, \right]^{\ixt}_{(1)}$
is the new top coefficient of their
difference.
Explicitly,
both
$$
\Upsilon_{k,(m-\ixt,\ixt)}(\phi)
\ \cdot\ 
\Upsilon_{k',(-\ixt, \ixt+m')}(\phi')
\qquad\text{ and }\qquad
\Upsilon_{k+k',(m-\ixt,\ixt+m')}(\phi\,\phi')
$$
are
$(\phi\,\phi')\, \ddel^{-k-k'}$
modulo lower order terms, i.e.,
modulo
$\poh_{-k-k'-1}$,
hence
\begin{equation}\label{FirstIterationStep}
\Upsilon(\phi)\cdot
\Upsilon(\phi')
-\Upsilon\big(\phi\cdot\phi'\big)\  
\ \ := \ \
\Upsilon_{k,(m-\ixt,\ixt)}(\phi)
\ \cdot\ 
\Upsilon_{k',(-\ixt, \ixt+m')}(\phi')
\ -\ 
\Upsilon_{k+k',(m-\ixt, \ixt+m')}(\phi\,\phi')
\end{equation}
is just 
$
\phi''
\ddel^{-k-k'-1} $
modulo $\poh_{-k-k'-2}$
for some unique $\phi'':=\left[\phi,\phi'\, \right]^{\ixt}_{(1)}$ 
in 
$\holos$:
$$
\left[\phi,\phi'\, \right]^{\ixt}_{(1)}
:=
\pi_{-k-k'-1}
\Big(
\Upsilon_{k,(m-\ixt, \ixt)}(\phi)\cdot
\Upsilon_{k', (-\ixt, \ixt+m')}(\phi')
\ -\ 
\Upsilon_{k+k', (m-\ixt,\ixt+m')}
\,
\left[\phi,\phi'\right]^{\ixt}_{(0)}\Big)\ 
\, .
$$
Using $\tree$, we lift 
$[\phi,\phi']_{(1)}^{\ixt}$
to a pseudodifferential operator 
whose top term cancels with that of
\cref{FirstIterationStep}
and set $[\phi,\phi']_{(2)}^{\ixt}$
to be the new top coefficient of the difference.
We continue inductively,
subtracting off 
$\tree_{k+k'+i, (m-\ixt, \ixt + m')}(
[\phi,\phi']_{(i)}^{\ixt})$
at each step of the recursion,
using
the projection map
$\pi$
to pick off the top coefficient
at each step.
The proof of
\cref{EquivariantIsomorphism}
thus gives the next corollary.

\begin{cor}
For Jacobi
forms $\phi, \phi'$,
the Rankin-Cohen bracket may be given iteratively
by
$$
\left[\phi,\phi'\right]_{(\nu)}
\ :=\ \pi\Big(\
\Upsilon(\phi)\cdot\Upsilon(\phi')
-\sum_{j=0}^{\nu-1} \Upsilon\big( \left[\phi,\phi'\right]_{(j)}\big)\Big)
\, ,
$$
i.e.,
for 
$\phi, \phi'$ of
respective weights $k, k'\geq 2$ in $\ZZ$
and indices $m, m'> 0$ in $\ZZ$, 
and any parameter $\ixt$ in $\CC$,
$$
\left[\phi,\phi'\right]^{\ixt}_{(\nu)}
\ :=\ 
\pi\Big(
\Upsilon_{k,(m-\ixt, \ixt)}(\phi)\cdot
\Upsilon_{k', (-\ixt, \ixt+m')}(\phi')
\ -\ \sum_{j=0}^{\nu-1} 
\Upsilon_{k+k'+j, (m-\ixt,\ixt+m')}
\,
\left[\phi,\phi'\right]^{\ixt}_{(j)} \Big)\ 
$$
for $\pi$
the projection
onto the coefficient of $\ddel^{-k-k'-\nu}$.

\end{cor}

\vspace{2ex}

\begin{remark}
{\em
We may 
rephrase \cref{BracketOfJacobiForms}
using this iterative approach:
For Jacobi forms $\phi$, $\phi'$,
$$
\Upsilon(\phi)
\ \Upsilon(\phi')\ =\ 
\sum_{\nu\geq 0}\Upsilon\big(\, [\phi,\phi']_{(\nu)}\, \big)
\, ,
$$
i.e., for $\phi$, $\phi'$ of
respective weights $k,k'\geq 2$ and indices
$m,m'> 0$ and
for any parameter $\ixt\in \CC$,
$$
\Upsilon_{k,(m-\ixt, \ixt)}(\phi)
\cdot 
\Upsilon_{k',(-\ixt, \ixt+m'})(\phi')
\ =\ 
\sum_{\nu\geq 0}\tree_{k+k'+\nu,\, (m-\ixt, \, \ixt+m')}\big(\, [\phi,\phi']^{\ixt}_{(\nu)}\, \big)
\, .
$$
}
\end{remark}

\vspace{2ex}


\subsection{Explicit formulas for the initial brackets}
\label{ExplicitBrackets}
To give the first and second brackets explicitly, we use the constants 
\begin{gather*}
  c_{n, k, \mlr}
:=\ 
  \frac{2 \mr (k+n - 1) (k+n - 2)}{ (\ml+\mr) n (3 - n - 2k)}
\ \tx{}
\end{gather*}
and abbreviate
$$
a_n=c_{n,k,(m-\ixt, \ixt)},
\ \ 
a'_n=c_{n,k',(-\ixt, \ixt + m')},
\ \ 
b_n=c_{n,k+k',(m-\ixt, \ixt+m')},
\ \ 
d=c_{1,k+k'+1,(m-\ixt, \ixt+m')}
\, .
$$
Then for Jacobi forms $\phi$ and $\phi'$ of respective weights $k, k'\geq 2$ and indices
$m,m'>0$, 
the first bracket is 
$$
\begin{aligned}
[\phi, \phi']_{(1)}^{\ixt}
\ =\ \
&
\ \big(
a'_1- b_1-k
\big)\
\phi \, \tfrac{\del \phi'}{\del z}
+
\big(
a_{1}-b_{1}
\big)\
\phi'\, \tfrac{\del \phi}{\del z}
\\
\ =\ \
&
-
\sfrac{km'(m-\ixt)+k'm (m'+\ixt)}
{m'(m+m')}
\ \phi \, \tfrac{\del \phi'}{\del z}
+
\sfrac{k m'(m-\ixt)+ k'm(m'+\ixt)}{m(m+m')}
\ \phi' \, \tfrac{\del \phi}{\del z}
\\
\ =\ \
&
-\sfrac{
k m'(m-\ixt)+ k'm(m'+\ixt)}{m m'(m+m')}
\big(
m\ \phi \, \tfrac{\del \phi'}{\del z}
+
m'\ \phi'\, \tfrac{\del \phi}{\del z}
\big)
\, .
\end{aligned}
$$
The second bracket is
$$
\begin{aligned}
[\phi,\phi']^{\ixt}_{(2)}
\ \ \ =\ \ \ \ \ 
&  
\big( \tfrac{1}{2}
k(k+1)
-ka'_1
+a'_2 (a'_1 + \tfrac{1}{2})
- b_2( b_1+\tfrac{1}{2})
- d(a'_1-b_1-k)
\big)\
\phi\  \tfrac{\del^2 \phi'}{\del z^2}
\\
+ &
\big( 
a_2(a_1 + \tfrac{1}{2})
- b_2 (b_1+\tfrac{1}{2})
-d(a_1-b_1)
\big)\
\phi'\ \tfrac{\del^2 \phi }{\del z^2}
\\
+ &
\big( 
a_1 a'_1 - a_1(k+1)
- b_2 ( 2b_1 +1)
- d(a_1+a'_1-2b_1-k)
\big)\
\tfrac{\del \phi}{\del z}
\,
\tfrac{\del \phi'}{\del z}
\\
+ &
4\pi i\big( 
 \ixt a'_2
+(m-\ixt)b_2
 \big)\
\phi \ \tfrac{\del \phi'}{\del \tau}
\\
+ &
4\pi i
 (\ixt-m)
( a_2-b_2 )\
\phi' \ \tfrac{\del \phi}{\del \tau}
\, .
\end{aligned}
$$
In
terms of the heat operator 
$\bbL_{m}
=8\pi i m \tfrac{\del}{\del \tau}
- \tfrac{\del^2}{\del z^2}$,
this is just
$$
\begin{aligned}
[\phi,\phi']^{\ixt}_{(2)}
\ \ \ =\ \ \ \ \ 
&  
\big( \tfrac{1}{2} k(k+1)
-ka'_1
+ a'_1  a'_2 - b_1 b_2
- d(a'_1-b_1-k)
\big)\
\phi\  \tfrac{\del^2 \phi'}{\del z^2}
\\
+ &
\big( 
a_1 a_2 
- b_1 b_2
-d(a_1-b_1)
\big)\
\phi'\ \tfrac{\del^2 \phi }{\del z^2}
\\
+ &
\big( 
a_1 a'_1 - a_1(k+1)
- 2 b_1 b_2
- d(a_1+a'_1-2b_1-k)
\big)\
\tfrac{\del \phi}{\del z}
\,
\tfrac{\del \phi'}{\del z}
\\
- &
\tfrac{1}{2}\, a'_2\,  \phi \ \bbL_{-\ixt}(\phi')
-
\tfrac{1}{2}\, a_2\, \phi'\ \ \bbL_{m-\ixt}(\phi)
+
\tfrac{1}{2}\, b_2\ \bbL_{m-\ixt}(\phi\, \phi')
\, .
\end{aligned}
$$

\vspace{2ex}

\appendix
\section{Jacobi Lie algebra
and finding the Casimir operator}
\label{FindingCasimir}
We explain here the theory behind
the construction of the Casimir operator $\mathcal{C}_{k,\mlr}$
in \cref{CasimirDefn}.
We use the action of the Jacobi Lie algebra $\mathfrak{g}$
and identify covariant operators
used to construct 
$\mathcal{C}_{k,\mlr}$ from a Casimir
element $\Omega$ in the universal enveloping algebra $\mathcal{U}(\mathfrak{g})$
for the complexified Lie algebra $\mathfrak{g}$ of the real Jacobi Lie group $G$.

\subsection{Lie algebra of the extended real Jacobi group}
We identify $\frakgJ$ with the
tangent space of $G$ at the identity so that
$$
\frakgJ
=
\big\{(M,X,\kappa)
:M\in\frak{sl}_2(\CC), X\in\CC^2,
\kappa\in \CC\big\}\, 
$$
with Lie bracket given by (see~\cite{BringmannConleyRichter})
$$
\big[(M,X,\kappa),(M',X',\kappa') \big]
=
\big(MM'-M'M,\ XM'-X'M,\
2 \det
\big[\begin{smallmatrix}X\\X'\end{smallmatrix}
\big]
\big)\, 
$$
and exponential map 
given in Conley and Raum~\cite{ConleyRaum}.
Here $\frakgJ \cong 
\frak{sl}_2(\CC)\ltimes {\frak H}_3$
where ${\frak H}_3$ is the complexified Lie
algebra of the real Heisenberg Lie group
$H_3(\RR)$. 
(Note that Bringmann, Conley, and Richter~\cite{BringmannConleyRichter}
consider the Lie algebra of the extended complex
Jacobi Lie group instead.)
The Lie algebra $\frakgJ$ has $\CC$-basis
$$
\begin{aligned}
   E &=\big(\big[\begin{smallmatrix} 0 &1\\ 0&0 \end{smallmatrix}\big],(0, 0),0\big)\, ,
&\ F &=\big(\big[\begin{smallmatrix} 0 &0\\ 1&0 \end{smallmatrix}\big], (0,0),0\big)\, ,
&\ H &=\big(\big[\begin{smallmatrix} 1 &\ \, 0\\ 0&-1 \end{smallmatrix}\big], (0,0),0\big)\, ,
\\
     e &=\big(\big[\begin{smallmatrix} 0 &0\\ 0&0 \end{smallmatrix}\big],(0,1),0\big)\, ,
& \  f &=\big(\big[\begin{smallmatrix} 0 &0\\ 0&0 \end{smallmatrix}\big],(1,0),0\big)\, ,
& \  Z &=\big(\big[\begin{smallmatrix} 0 &0\\ 0&0 \end{smallmatrix}\big],(0,0),1\big)\, .
\end{aligned}
$$

\vspace{1ex}

\subsection{Alternate basis of Lie algebra}
\label{AlternateBasis}
We use an alternate $\CC$-basis of $\frakgJ$
from~\cite{ConleyRaum} 
(see~\cite{BerndtSchmidt})
to define covariant operators
(as these span the characters
of the compact group
$\mathrm{SO}_2(\RR)$
embedded in the Jacobi group
acting via the adjoint representation).
Define
\begin{equation*}
\begin{aligned}
  \tilde{E}
 = \tfrac{1}{2} \big( H + i (E + F) \big)
\text{,} \ \ 
\tilde{e}
=  \tfrac{1}{2} \big(f+i e  \big)
\text{,}\ \ 
\tilde{H}
= i (F - E)
\text{,}\ \ 
\tilde{F}
= \tfrac{1}{2} \big( H - i (E + F) \big)
\text{,}\ \ 
\tilde{f}
 = \tfrac{1}{2} \big(f-i e  \big)
\text{,}\ \ 
\tilde{Z} = \tfrac{1}{2}iZ
\text{.}
\end{aligned}
\end{equation*}
\vspace{1ex}

\subsection{Lie algebra action}
%
Suppose the connected Lie group $G$ 
acts smoothly on the right on some complex vector space $W$
with action denoted via a slash operator:
$(w,g) \mapsto w \, \big|\, g$
for $w$ in $W$, $g$ in $G$.
Recall that the {\em differential action}
on $W$ of the complexified Lie algebra $\frak{g}$
of $G$ given by 
\begin{gather*}
   w\ \big|\ X 
\ =\ \tfrac{d}{ds}\Big|_{s=0}\big( w\ \big| \ \exp(sX) \big)
\quad\quad\text{ for } 
w \in W,\ X \in \frak{g}
\, 
\end{gather*}
defines a Lie algebra representation 
of $\mathfrak g$. 
We compute the explicit
action of $\mathfrak g$
on $\po$
at the marked point $(i,0)$ of $\HJ$
to simplify calculations.
These
generalize actions given in 
Berndt and Schmidt~\cite[Section 3.5]{BerndtSchmidt}
on $C^{\infty}(\HJ)$.

\begin{lemma}
\label{la:lie-algebra-action-at-origin}
Fix some $k,\ml, \mr$ in $\CC$.  
At $(\tau, z) = (i, 0)$,
for slash action $|_{k, \mlr}$
denoted simply by $|\, $,
\begin{small}
\begin{align*}
\phi\, \ddel^n\ | \ E
&= 
  (\phi_{\tau} + \phi_{\ov \tau})\, \ddel^n
\,,
&
\phi\, \ddel^n\ | \ F
&=
 \phi_\tau + \phi_{\ov \tau} + i (n - k) \phi 
\, \ddel^n
  - 2 \pi i\, \mr n(n-1) \phi\, \ddel^{n-2}
\,,\\
\phi\, \ddel^n\ | \ e
&=
  (\phi_z + \phi_{\ov z})\, \ddel^n
\,,
&
\phi\, \ddel^n\ | \ f
&=
  ( i \phi_z - i \phi_{\ov z} )\, \ddel^n
  + 4\pi i\, \mr n\, \phi\, \ddel^{n-1}
\,,\\
\phi\, \ddel^n\ | \ H
&=
  \big( 2 i \phi_{\tau} - 2 i \phi_{\ov \tau} - (n - k) \phi \big)\, \ddel^n
\,,&
\phi\, \ddel^n\ | \ Z
&=
  2\pi i (\ml+\mr) \phi\, \ddel^n
\text{.}
\end{align*}
\end{small}
\end{lemma}

Direct computation 
using \cref{la:lie-algebra-action-at-origin} gives the following
actions for the alternate basis
of $\mathfrak{g}$.
\begin{lemma}
\label{lemma:lie-algebra-action-at-origin-raising-and-lowering-elements}
Fix some $k,\ml, \mr$ in $\CC$.  
At $(\tau, z) = (i, 0)$,
for slash action $|_{k, \mlr}$
denoted simply by $|\, $,
\begin{small}
\begin{align*}
  \phi\, \ddel^n\, | \ \tilde{E} 
&=
  \big( 2 i \phi_\tau + (k - n) \phi \big)\, \ddel^n + \pi \mr n (n - 1) \phi\, \ddel^{n-2}
\text{,}
&& \phi\ \ddel^n \, | \ \tilde{F} 
=  - 2 i \phi_{\ov \tau}\, \ddel^n
  - \pi \mr n (n - 1) \phi\, \ddel^{n-2}
\text{,}
\\
\phi\, \ddel^n\, | \ \td{e} 
&=
  i \phi_z\, \ddel^n
  + 2 \pi i \mr n \, \phi\, \ddel^{n-1}
  \text{,}
&& \phi\ \ddel^n \, | \ \, \td{f} 
\, = 
  - i \phi_{\ov z}\, \ddel^n
  + 2 \pi i \mr n \, \phi\, \ddel^{n-1}
\text{,}
\\
\phi\, \ddel^n\, | \ \td{H} 
&=
  (k - n) \phi\, \ddel^n
  + 2 \pi \mr n (n  - 1)\, \phi\, \ddel^{n-2}
  \text{,}
&& \phi \, \ddel^n\, | \ \td{Z} 
 = 
  -\pi  (\ml+\mr) \phi\, \ddel^n
\, \text{.}
\end{align*}
\end{small}
\end{lemma}

\vspace{1ex}

\subsection{Covariant differential operators}
We follow the strategy of~\cite{ConleyRaum}
(see also Conley and Dahal~\cite{ConleyDahal})
which extends
Helgason's treatment 
(see~\cite{Helgason1977}) 
of the reductive case
to the nonreductive Jacobi group.
We consider operators analogous
to the covariant differential operators
$\rm{CDO}_{\gamma, \gamma'}$
constructed in~\cite{ConleyRaum}
and~\cite{BringmannConleyRichter}.
For the next definition, we fix
a choice of section 
$s:\HJ\rightarrow \GJ$ 
to the map $\GJ\rightarrow \HJ$, $g\mapsto g(i,0)$,
with 
$s(i,0)=1_G$ 
the identity,
so $s(\tau, z)(i,0)=(\tau,z)$.
We use the adjoint action of $g$ in $G$ on $X$ in $\mathfrak{g}$
denoted by $gXg^{-1}=\text{ad}_g(X)
=\tfrac{d}{dt}\big|_{t=0} \,g \exp(tX)
g^{-1}
\, .
$
\begin{definition}{\em
For any $X$ in $\frakgJ$, define an operator
\begin{gather*}
\begin{aligned}
\mathcal{D}_{k,\mlr}(X):
\ \ \ 
\po
&\rightarrow \po 
\qquad\qquad\qquad\qquad\text{ by }
\\
\phi(\tau,z)\, \ddel^n
\ \ &\mapsto \ \
\big(\phi\, \ddel^n \ \ \big|_{k,\mlr}\, gXg^{-1} \big)(\tau,z)
\end{aligned}
\end{gather*}
for all $(\tau,z)\in \HJ$ and
$\phi\in C^{\infty}(\HJ)$,
for $g=s(\tau,z)\in\GJ$
(so that $g(i,0)=(\tau, z))$.
}
\end{definition}

\vspace{2ex}

\begin{proposition}
\label{AltBasisIsRaisingLowering}
The raising, lowering, and constant operators on $\po$ of \cref{def:raisingandlowering}
arise from the alternate basis 
$\td{E},\td{F},\td{e},\td{f},\td{H},\td{Z}$
of the Lie algebra $\frakg$:
For any $k, \ml, \mr$ in $\ZZ$,
with subscript
${k, \mlr}$ suppressed,
$$
\begin{aligned}
  &\mathcal{D}(\td{E})\ 
  &=&   
  &&\Rmod  \, ,
\quad  
&&\mathcal{D}(\td{F})\ 
&=&
&&\Lmod  \ ,
\\
  &\mathcal{D}(\td{e})\ 
  &=&
  &&\Rell \, ,
&&\mathcal{D}(\td{f})\ 
&=& 
&&\Lell \, ,
\\
  &\mathcal{D}(\td{H})\ 
  &=&
  &&\Cmod \,  ,
&&\mathcal{D}(-2i\td{Z})\ 
&=& 
&&\Cell
 \, 
\end{aligned}
$$
as operators $\po\rightarrow \po$.
\end{proposition}
\begin{proof}
We use results from~\cite{ConleyRaum}
describing covariant operators.
Although the Jacobi Lie algebra $\frakgJ$ is not reductive,
$\HJ\cong G/K$ is a {\em reductive coset space} 
(as described by Helgason~\cite{Helgason1959, Helgason1977})
for (noncompact)
$K$ the stabilizer
in $G$ of $(i,0)$ in $\HJ$
since
$\frakgJ=\frakmJ\oplus\frakkJ$
with $\frakkJ$ the complexified Lie algebra of $K$ and 
$\frakmJ$ a Lie subalgebra invariant under the action
of $\mathrm{Ad}_{\GJ}(K)$:
Here (see~\cite{ConleyRaum},~\cite{BringmannConleyRichter}, and 
~\cite{BerndtSchmidt}),
\begin{gather*}
\begin{aligned}
K&=\big\{(M,0,\kappa):M\in\text{SO}_2(\RR),\kappa\in \RR\big\},\\
\frak k&=\big\{(M,0,\kappa):M\in\frak{so}_2(\CC), \kappa\in \CC\big\}
= \CC\text{-span}\{\td{H}, \td Z \}, \\
\frak m&=\big\{(M,X,0):M\in\text{sl}_2(\CC),M=M^t, 
X\in \CC^2\big\}
= \CC\text{-span}\{\td{E}, \td{F}, \td{e}, \td{f} \}, 
\end{aligned}
\end{gather*}
with
$\tilde{H}$ 
acting diagonally on $\frak m$ with respect to an
eigenbasis 
$\tilde{E}, \tilde{F}, \tilde{e}, \tilde{f}$
with respective $\tilde{H}$-weights $2,-2,1,-1$.
We view $\po$ as 
$C^{\infty}(G)\ot_K V$ 
for ring of Laurent series
$V=\CC(\!(\ddel^{-1})\!)$.

By \cref{cor:covariantoperators1},
$\Rmod$
is a covariant operator
that
raises the weight by $2$,
see \cref{RaisingLowering}.
The operator 
$\mathcal{D}(\tilde E)$
also defines a covariant operator
that raises the weight by $2$
by Corollary 5.8
of~\cite{ConleyRaum},
see the argument in Section~5.4 there.
In other words,
both operators lie in
the vector space of equivariant maps 
$$
\Big(\po,\ |_{k, \mlr}\Big)^G
\longrightarrow 
\Big(\po,\ |_{k+2, \mlr}\Big)^G.
$$
By~\cite[Corollary 5.8]{ConleyRaum},
this
space of covariant operators
is $1$-dimensional
over $\CC$,
and thus $\Rmod$ and
$\mathcal{D}(\tilde E)$
agree
up to a constant.
The operator 
$\mathcal{D}(\tilde E)$
at the point $(\tau,z)=(i,0)$
in $\HJ$
is just
slashing by $\tilde E$
since we take the adjoint action of $g=1_G$ on $\tilde E$.
Upon inspection using 
\cref{lemma:lie-algebra-action-at-origin-raising-and-lowering-elements},
we see that
$(\, \textendash \ | \ \tilde{E})$ and
$\Rmod$ agree
at the point $(i, 0)$, hence this constant is $1$
and $\Rmod = \mathcal{D}(\tilde{E})$. 
Analogous reasoning applies to 
each of the
other indicated operators using
the representation theory
arguments of~\cite{ConleyRaum}
and~\cite{BringmannConleyRichter}.
\end{proof}

\vspace{2ex}

\begin{remark}{\em
In practice, to compute the
covariant operators $\mathcal{D}(X)
=\mathcal{D}_{k, \mlr}(X)$
for Lie algebra elements $X=\td{E},\td{F},\td{H},\td{e},\td{f},\td{Z}$
that give the raising and lowering operators 
$\Rmod$,
$\Lmod$,
$\Cmod$,
$\Lell$,
$\Rell$,
$\Cell$,
we use the section $s:\HJ\rightarrow G$ given by
$$
(\tau, z)\mapsto g
=
  \tfrac{1}{\sqrt{y}} \Big(
  \left(\begin{smallmatrix} 
y\ \ & \ x\\
0\ \ & 1\rule{0ex}{1.5ex} \end{smallmatrix}
\right)
  ,\;
(v, \ \ u ),\ 0
  \Big)
\in G\,
\text{,}
$$
for $\tau=x+iy$, $z=u+iv$,
so that $g(i, 0) = (\tau, z)$.  
We compute 
$
\Big(\big(\phi \delz^n\, \big|\, g\big)\, \big|\, X\Big)(i,0)
$
for fixed $(\tau,z)$
and then find the functions $\phi_j':\HJ \rightarrow \CC$ 
(depending on $X$) for which
$$
\Big(\big(\phi\, \delz^n\, \big|\, g\big)\, \big|\, X\Big)(i,0)=
\big( \sum_j \phi_j' \ \delz^{n-j}\ \big)(i,0).
$$
After
slashing with $g^{-1}$, we may  evaluate
at $(\tau, z)$
to obtain
$\big(\mathcal{D}(X)\phi \, \delz^n\big)(\tau,z)
= 
\big( \sum_j \phi_j' 
\ \delz^{n-j}\, \big| \, g^{-1}\big)(\tau,z)
\, .
$
}
\end{remark}
\vspace{1ex}

\subsection{Casimir element}

Our proof of the existence of a splitting
map $\tree$as in \cref{SplittingExists}
relied upon construction of an
operator on $\po$
that acts by a scalar
on each graded piece.
Hence we look at the center
of the universal enveloping algebra
with an eye to Schur's Lemma.
We explain here how we use the Casimir element
for the extended complex
Jacobi Lie algebra $\frak g$
to find $\mathcal{C}_{k,\mlr}$.

The center, $Z(\mathcal{U}(\frak g))$,
of the universal enveloping algebra
of $\frak g$ is generated
by two elements: $Z$ and 
$$
\Omega :=
ZH^2-3ZH+4ZEF-Hef+2ef+Ef^2-e^2F\, 
$$
(see~\cite{ConleyRaum}).
The element $Z$ 
merely acts by a scalar
multiple of the index
on $\holos$ and $\poh$,
see \cref{la:lie-algebra-action-at-origin}.
Hence
we regard $\Omega$ as 
the {\em Casimir element}
and rewrite it in terms of the alternate basis of 
$\frakgJ$ (see~\cref{AlternateBasis})
using 
a Gr\"obner basis computed
with the help of software extension Plural~\cite{Plural} of Singular~\cite{Singular}:
\begin{gather}
\Omega\ =\ 
   -8 i \td{Z} \td{E} \td{F}
  + 2 i \big( \td{e}^2 \td{F} - \td{E} \td{f}^2 \big)
  + 2 i \td{e} \td{f} ( \td{H} - 2)
  + 6 i \td{Z} \td{H} 
  - 2 i \td{Z} \td{H}^2
\, .
\end{gather}
(Compare to Proposition~2.7 
and the remark before the proof of Proposition~2.6
in~\cite{ConleyRaum}.)

\vspace{1ex}

\subsection{Casimir operator}
As $\po$
is a Lie algebra representation,
we may use
the Poincar\'e-Birkhoff-Witt Theorem to extend the function 
\begin{gather*}
\mathfrak{g}\times 
\po\longrightarrow\po,
\quad
(X,\,  \phi\, \ddel^n)
\mapsto 
\mathcal{D}_{k,\mlr}(X)
(\phi\, \del^n)
\end{gather*}
to a linear map
\begin{gather*}
\mathcal{U}(\frak g)\times \po \longrightarrow \po,
\end{gather*}
i.e., a linear map from 
$\mathcal{U}(\frak g)$
to
the ring of operators on $\po$
under composition.
We just choose an ordering
of the tilde basis of $\frak g$,
say 
$\td Z<\td E<\td e< \td F < \td f< \td H$,
and fix a basis for $\mathcal{U}(\frak g)$ 
of monomials in these elements in this order.
The Casimir element $\Omega$ 
then maps to a {\em Casimir operator}
\begin{gather*}
\mathcal{C}_{k,\mlr}
\ :=\ 
\ \ \po\longrightarrow\po\ 
\end{gather*}
given as the composition of 
order $1$ operators
(applied left to right):
\begin{equation*}
\begin{aligned}
\mathcal{C}_{k, \mlr}
\ =\ 
& - 8 i \,\D(\td{Z})\, \D(\td{E})\, \D(\td{F})
  + 2 i \,\D(\td{e})^2\,\D(\td{F}) 
  - 2 i \,\D(\td{E})\,\D(\td{f})^2 \\
& + 2 i \,\D(\td{e})\,\D(\td{f}) \,(\D(\td{H}) - 2)
  + 6 i \,\D(\td{Z})\,\D(\td{H}) 
  - 2 i \,\D(\td{Z})\,\D(\td{H})^2
  \, ,
\end{aligned}
\end{equation*}
giving \cref{CasimirDefn}
by \cref{AltBasisIsRaisingLowering}.
One may employ an argument
using a vector bundle
over the homogeneous
space $G/K$ for $K$ the stabilizer
in $G$ of $(i,0)$ in $\HJ$
reminiscent of that in~\cite{ConleyRaum,ConleyDahal}
to argue that the
operator $\mathcal{C}=\mathcal{C}_{k, \mlr}$ 
is $G$-equivariant
for all $k$ in $\ZZ$ and $\ml, \mr$  in $\CC$,
see \cref{CasimirEquivariant}.
We have omitted this lengthy argument
since the covariance
of the order $1$ operators
can be checked directly.

\vspace{2ex}

\begin{remark}{\em 
Our Casimir operator
is not the usual Casimir operator
determined by the Casimir element
in historic Lie theory terms.
Traditionally, 
the Casimir operator would be
found using Helgason's theory
of differential operators 
(see~\cite[Theorem~10]{Helgason1959}) 
which gives a {linear}
map from $\mathcal{U}(\mathfrak{g})$ to
the set of
operators on $\po$
using {\em symmetrization} combined with
partial differentiation,
\begin{gather*}
  \phi\, \ddel^n\ \big|\ [Y_1 \cdots Y_l]
\ =\ 
  \tfrac{d}{ds_1}\, \Big|_{s_1 = 0} \cdots \
  \tfrac{d}{ds_l}\, \Big|_{s_l = 0}\ 
  \Big(\phi\, \ddel^n\ \big|\ e^{s_1 Y_1 +\cdots+ s_l Y_l}\Big)
  \, 
\quad\text{ for } Y_i\in\frak{g}\, .
\end{gather*}
We
instead use a {\em composition} of operators of
degree $1$,
in parallel to 
a customary approach
in modular forms.
}\end{remark}

\section{
Noncommutative multiplication
of pseudodifferential operators}
\label{AppendixB}

We include the advertised
closed formula 
in \cref{SlashActionSection}
for recording the noncommutative
multiplication of 
pseudodifferential operators.

\vspace{1ex}

{\bf \cref{triplesum}}.
In the noncommutative ring $\, \po$,
for all $c_1,c_2 \in \CC$ and $n \in \ZZ$,
\begin{gather*}
  (\ddel + c_1 z + c_2)^n
\ =\ \
  \sum_{0\leq p}\
  \sum_{\substack{ 0\, \le\, t, s, r \\ t+s+2r=p\rule{0ex}{1.5ex}}}
  2^r \,
  \sbinom{n}{t, s, 2r, n-p}
\  \sfrac{\Gamma\big(r + \tfrac{1}{2}\big)\rule[-1ex]{0ex}{3ex}}
{\Gamma\big(\tfrac{1}{2}\big)\rule[0ex]{0ex}{2.5ex}}
\ \  c_1^r\, c_2^s\ (c_1 z)^t\ \ddel^{n - p}
\text{.}
\end{gather*}

\vspace{1ex}

\begin{proof}
We first show the formula for $n\geq 0$
via a generating series
in~$n$ using the noncommutative
multiplication in the ring of regular differential
operators 
(see \cref{eq:multiplication-of-pseudo-differential-operators}) to move
all $\ddel$ to the far right:
\begin{gather*}
  (\ddel + c_1 z + c_2)^n                                                                         
\ =
  \sum_{\substack{ 0 \le j, t, s, r,  \rule{0ex}{1.5ex}
\\ j+ t+s + 2 r  = n\rule{0ex}{1.5ex} }}
  \ c_{j, t, s, r} \
  c_1^r\, c_2^s\, (c_1 z)^t\, \ddel^j\ 
\end{gather*}
for some complex scalars $c_{t,j,r,s}$.
We encode these coefficients
in the  generating series
\begin{gather*}
  F(u, v, x, y)\
=
  \sum_{0\leq j, t, s, r} c_{j, t, s, r}\, u^r v^s x^t y^j
\quad\text{with commuting complex variables}\quad u, v, x, y
\, 
\end{gather*}
with an initial condition
$F(0, v, x, y) = (1-v-x-y)^{-1}$,
as
contributions from $r = 0$ 
can be computed as in a
commutative ring.
We may regard $F$ as a formal
expression in noncommutative variables 
so that 
$$
F(c_1,c_2,c_1z,\ddel)
=  \sum_{0\leq n} (\ddel+c_1z+c_2)^n
\, 
$$
and 
expand
\begin{gather*}
F(c_1,c_2,c_1z,\ddel)-1
=
\sum_{0\leq n}(\ddel+c_1z+c_2)^{n+1}
=
(\ddel+c_1z+c_2)\sum_{0\leq n}(\ddel+c_1z+c_2)^{n}\ 
\end{gather*}
as
\begin{align*}
  (\ddel  + c_1 z  + c_2)
  & \sum_{\substack{ 0 \le j, t, s, r\rule{0ex}{2ex} \\ 
j + t+ s+ 2 r = n \rule{0ex}{2ex} }}
  c_{j, t, s, r} \
  c_1^r c_2^s (c_1 z)^t \ddel^j
\\
={}&
  \sum_{\substack{ 0 \le j, t, s, r \rule{0ex}{2ex} \\ 
 j + t+ s+ 2 r = n\rule{0ex}{2ex} }}
  c_{j, t, s, r}\ 
  \rule{0ex}{3ex}
  \big(   c_1^r c_2^s (c_1 z)^t \ddel^{j+1}
        + t c_1^{r+1} c_2^s \, (c_1 z)^{t - 1} \ddel^j
        + c_1^r c_2^s\, (c_1 z)^{t+1} \ddel^j
        + c_1^r c_2^{s+1}\, (c_1 z)^t \ddel^j
  \big)\,\text{.}
\end{align*}
We return to the commutative setting
and use the Poincar\'e-Birkhoft-Witt property to
rewrite this last expression 
as
\begin{gather*}
yF(u,v,x,y)
+xF(u,v,x,y)
+vF(u,v,x,y)
+u\tfrac{\del}{\del x}F(u,v,x,y)
\end{gather*}
formally evaluated at $(u,v,x,y)=(c_1,c_2,c_1z, \ddel)$.
We conclude that
\begin{gather*}
F(u,v,x,y)-1
=\big(y+x+v+u\tfrac{\del}{\del x}\big)\ F(u,v,x,y)\ 
\end{gather*}
and 
$ (1 - x - y - v - u \tfrac{\del}{\del x})\, 
F(u, v, x, y)
=  1$.

To find $F$, 
we find
a power series 
${\td F}(u, w)$
in $u$ with coefficients that are Laurent series in
$w$
solving the differential equation
$  (w + u \partial_{w}) {\td F}
(u, w)
=1$
with initial condition
${\td F}(0, {w}) = {w}^{-1}$.
We then obtain $F$
by setting
$F(u,v,x,y)=
{\td F}(u, 1 - x - y - v)$.
We write
$\td{F}(u,{w})=  
\sum_{0\leq r} h_r({w})\, u^r
$,
and we notice that $h_0({w}) = {w}^{-1}$ and 
${w} h_r({w}) 
+ \partial_{w} h_{r - 1}({w}) = 0$ 
(for $r > 0$).
We conclude that
\begin{gather*}
  h_{r}(w)=
  \ 2^r \ 
\sfrac{\Gamma\big(r + \tfrac{1}{2}\big)}
{\Gamma\big(\tfrac{1}{2}\big)\rule{0ex}{2.5ex}} \
w^{-1-2r}\ .
\end{gather*}
Then since $n = j+t  + s + 2 r$,
\begin{align*}
  F(u, v, x, y)\
&=\
  \sum_{0 \le r}
  \ \ 2^r \ \ 
\sfrac{\Gamma\big(r + \tfrac{1}{2}\big)}
{\Gamma\big(\tfrac{1}{2}\big)\rule{0ex}{2.5ex}} \,
\  \sfrac{u^r}{(1 - x - y - v)^{1 + 2 r}}
\\
&=
  \sum_{0 \le t, j, r, s}
  (-1)^{t + j + s} \ 2^r \
  \sbinom{-1 - 2r}{t, j, s, -1-n}
  \ \sfrac{\Gamma\big(r + \tfrac{1}{2}\big)}{\Gamma\big(\tfrac{1}{2}\big)
\rule{0ex}{2.5ex}} \
  u^r v^s x^t y^j\
\text{.}
\end{align*}
The second equality follows from generalized multinomial
series:
Newton's binomial series
$(1-\td w)^{\alpha} = \sum_{0\le k} \binom{\alpha}{k}\,(-\td w)^k$
holds for $\alpha \in \CC$ and $|\td w|<1$ (see~\cite{Knopp}),
and we set $\td w = x+y+v$ 
to express
$(1-x-y-v)^{-1-2r}$ via the multinomial coefficients
$\sbinom{-1-2r}{t, j, s, -1-n}$,
see \cref{multinomial}
and~\cite[(26.4.2)]{DLMF}.
Euler's reflection formula for the Gamma function implies that
\begin{gather*}
  (-1)^{t + j + s} \ \sbinom{-1 - 2r}{t, j, s, -1-n}
=
  \sbinom{n}{t, j, s, 2r}
\ ,
\end{gather*}
and hence in the ring of ordinary
differential operators
\begin{gather*}
  (\ddel + c_1 z + c_2)^n
\ =\ \
  \sum_{0\leq p\leq n}\
  \sum_{\substack{ 0\, \le\, t, r, s \\ t+s+2r=p\rule{0ex}{1.5ex}}}
  2^r \,
  \sbinom{n}{t, s, 2r, n-p}
\  \sfrac{\Gamma\big(r + \tfrac{1}{2}\big)\rule[-1ex]{0ex}{3ex}}
{\Gamma\big(\tfrac{1}{2}\big)\rule[0ex]{0ex}{2.5ex}}
\ \  c_1^r\, c_2^s\ (c_1 z)^t\ \ddel^{n - p}
\text{.}
\end{gather*}

To obtain a description of
$(\ddel + c_1 z + c_2)^n$
when $n<0$,
we observe that
the coefficients 
in a similar formula must be polynomials in~$n$,
which follows by recursion after expanding~$(\ddel + c_1 z + c_2)^n (\ddel + c_1 z + c_2)^{-n} = 1$.
We may obtain this
formula for $n<0$ by
extending the last sum
given for $n\geq 0$
to a sum over all $0\leq p$
noting that the multinomial coefficient
is zero for $p> n$,
and then, 
for each fixed $p$,
extending the coefficients
on the right hand side as polynomials in $n\geq 0$ to all $n$ in $\ZZ$.

\end{proof}


\vspace{3ex}

\bibliographystyle{amsalpha}

\begin{thebibliography}{999}


\bibitem{AthreyaLagaceMollerRaum}
J.\ Athreya,
J. Lagacé, M.\ Möller, and M.\ Raum, 
``Spectral decomposition and Siegel-Veech transforms for strata: the case of marked tori'', J.\ Spectr.\ Theory 15(2) (2025), 895--959.


\bibitem{AtiyahMacdonald}
M.\ F.\ Atiyah and I.\ G.\ Macdonald,
{\em Introduction to Commutative Algebra.} 
Addison-Wesley, Reading, Massachusetts, 1969. 

\bibitem{BerndtSchmidt}
R.\ Berndt and R.\ Schmidt,
{\em Elements of the representation theory of the Jacobi group}
[2011 reprint of the 1998 original]. 
Birkh\"auser/Springer, Basel AG, Basel, 1998. 

\bibitem{BieliavskyTangYao}
 P.\ Bieliavsky, X.\ Tang, and Y.\ Yao,
 ``Rankin-Cohen brackets and formal quantization'',
 Adv.\ Math.\  212 (2007), 293--314.

\bibitem{BlochOkounkov} 
S.\ Bloch and A.\ Okounkov,
``The character of the infinite wedge representation'', 
Adv.\ Math.\ 149 (2000), 1--60.


\bibitem{Boecherer} 
S.\ B{\"o}cherer,
``Bilinear Differential Operators for the Jacobi Group'',
Comm.\ Math.\ Univ.\ St.\ Pauli 47 (1998), 135--154.

\bibitem{BringmannConleyRichter}     
K.\ Bringmann, C.\ Conley, 
and O.\ Richter,
``Maass-Jacobi forms over complex quadratic fields'',
Math.\ Res.\ Lett.\ (2007) 14(1),
137--156. 

\bibitem{BringmannRaumRichter}     
K.\ Bringmann, M.\ Raum, 
and O.\ Richter,
``Harmonic Maass-Jacobi forms with singularities 
and a theta-like decomposition'', 
Trans.\ Amer.\ Math.\ Soc.\ 367(9) (2015), 6647--6670.


\bibitem{Choie} 
Y.\ Choie, 
``Jacobi forms and the heat operator'', 
Math.\ Z.\ 225 (1997), 95--101.



\bibitem{ChoieDumasMartinRoyer2017}
Y.\ Choie, F.\ Dumas, F.\ Martin, and E.\ Royer,
``Rankin-Cohen deformations of the algebra of Jacobi forms'' 
(hal-01673663v1), 2017.


\bibitem{ChoieDumasMartinRoyer}
Y.\ Choie, F.\ Dumas, F.\ Martin, and E.\ Royer,
``Formal deformations of the algebra of Jacobi forms and
Rankin-Cohen brackets'',
Comptes Rendus.\ Math\'ematique 359(4) (2021), 505--521.




\bibitem{ChoieEholzer} 
Y.\ Choie and W.\ Eholzer,
``Rankin-Cohen operators for Jacobi and Siegel
forms'',
 J.\ Number Theory
68(2) (1998), 160--177.
 

\bibitem{ChoieLee}
Y.\ Choie and M.\ H.\ Lee,
``Symmetric tensor representations, quasimodular forms, and weak 
Jacobi forms'',
Adv.\ Math.\ 287(10) (2016), 
567--599.

\bibitem{ChoieLeeBook}
Y.\ Choie and M.\ H.\ Lee,
``Jacobi-like Forms and Pseudodifferential Operators''.
In
{\em Jacobi-like forms, pseudodifferential operators, and quasimodular forms.}
Springer Monographs in Mathematics. 
Springer Monogr.\ Math.,
Springer, Cham.,
2019.

\bibitem{ChoieParkZagier}
Y.\ Choie, Y.\ Park, D.\ Zagier,  
``Periods of modular forms on $\Gamma_0(N)$ and products of Jacobi theta functions'',
 J.\ Eur.\ Math.\ Soc.\ 21 (2017),
 95--101.

 
\bibitem{CMZ} 
P.\ Beazley Cohen, Y.\ Manin, 
and D.\ Zagier,
``Automorphic pseudodifferential operators''. In
{\em Algebraic aspects of integrable systems}, 
Progr.\ Nonlinear Differential Equations Appl. 26, 
Birkhäuser Boston, Boston, MA, 1997,
17--47.

\bibitem{ConleyDahal}
C.\ Conley and R.\ Dahal, 
``Centers and characters of Jacobi group-invariant differential operator algebras'',
J.\ Number Theory 148 (2015), 40--61.

\bibitem{ConleyRaum}
C.\ Conley and M.\ Raum,
``Harmonic Ma\ss-Jacobi forms of degree 1 with higher rank indices'',
Int.\ J.\ Number Theory
12(07) (2016), 
1871--1897.

\bibitem{ConnesMoscovici}
A.\ Connes and H.\ Moscovici
``Rankin-Cohen brackets and the Hopf algebra of transverse geometry'',
Mosc.\ Math.\ J.\ 
4(1) (2004),
111--130.

\bibitem{Dijkgraaf} 
R.\ Dijkgraaf,
``Mirror symmetry and elliptic curves''.
In {\em The moduli space of curves},
Progress in Mathematics 129, 
Birkh\"auser, 1995, 149--163.





\bibitem{DLMF}
F. W. J. Olver, A. B. Olde Daalhuis, D. W. Lozier, B. I. Schneider, R. F. Boisvert, C. W. Clark, B. R. Miller, B. V. Saunders, H. S. Cohl, and M. A. McClain.
NIST Digital Library of Mathematical Functions. Release 1.2.6.
\url{https://dlmf.nist.gov/}.

\bibitem{DumasMartin}
F.\ Dumas and F.\ Martin, 
``Invariants of formal pseudodifferential operator algebras and algebraic modular forms'',
Revista de la Unión Matemática Argentina 65 (2023), 1--31.

\bibitem{DumasRoyer}
F.\ Dumas and E.\ Royer,
``Poisson structures and star products on quasimodular forms'',
Algebra Number Theory 8(5) (2014), 
1127--1149.


\bibitem{EichlerZagier}
M.\ Eichler and D.\ Zagier,
{\em
 The theory of Jacobi forms}, Progr.\ Math.\ 55, Birkhäuser, Boston and Basel, 1985.



\bibitem{Ittersum}
J.-W.\ M.\ van Ittersum,
``The Bloch--Okounkov theorem for
congruence subgroups and Taylor
coefficients of quasi-Jacobi 
forms'',
Res.\ Math.\ Sci.\ 10(5) (2023),
45 pp. 

 \bibitem{Helgason1959} 
S.\ Helgason, 
``Differential Operators on Homogeneous Spaces'', 
Acta.\ Math.\ 102 (1959), 
239--299.

\bibitem{Helgason1977} 
S.\ Helgason,
``Invariant differential equations on homogeneous manifolds'',
Bull. Amer. Math. Soc. 83(5) (1977), 
751--774.


\bibitem{KanekoZagier95}
M.\ Kaneko and D.\ Zagier, 
``A generalized Jacobi theta function and quasimodular forms''.
In {\em 
The moduli space of curves (Texel Island, 1994)}, 
Progress in Mathematics 129, Birkhäuser Boston, Boston, MA, 1995,
165--172.


\bibitem{Knopp}
K.\ Knopp,
{\em Theory and Application of Infinite Series},
2nd English ed., Blackie \& Son, London and Glasgow, 1951.

\bibitem{Kontsevich}
M.\ Kontsevich,
``Deformation Quantization of Poisson Manifolds'',
Letters in Mathematical Physics 66 (2003), 157--216.

\bibitem{MinHoLee} M.\ H.\ Lee,
``Casimir Operators on Pseudodifferential Operators of Several Variables'',
J.\ Lie Theory 12 (2002),
483--493.

\bibitem{OnoUnearthing} 
K.\ Ono, 
``Unearthing the visions of a master: harmonic Maass forms and number theory'',
Curr.\ Dev.\ Math.\ Sci.\ C
2008 (2008), 347--455.


\bibitem{Ovsienko} V.\ Ovsienko,
``Exotic Deformation Quantization'',
J.\ Differential Geometry, 45(2) (1997), 390--406.


\bibitem{Plural}
V.\ Levandovskyy, 
``Plural, a Non-commutative Extension of Singular: Past, Present and Future''. In  Mathematical Software - ICMS 2006, A.\ Iglesias and N.\ Takayama (Eds.), Vol.\ 4151, 
 Springer, Berlin, Heidelberg, 2006, 129--140.

\bibitem{Pitale} A.\ Pitale,
``Jacobi Maass Forms'', 
Abh.\ Math.\ Sem.\ Univ.\ 
Hamburg 79 (2009), 87--111.

\bibitem{ramanujan-1988}
S.\ Ramanujan,
{\em The lost notebook and other unpublished papers},
with an introduction by George E.\ Andrews,
Springer, Berlin, 1988.
 
\bibitem{Raum2015}
M.\ Westerholt-Raum, ``Harmonic Maaß-Jacobi forms of degree 1'', Mathematical Sciences 2, 12 (2015). 

\bibitem{Singular}
Decker, W.; Greuel, G.-M.; Pfister, G.; Sch{\"o}nemann, H.: 
\newblock {\sc Singular} {4-4-0} --- {A} computer algebra system for polynomial computations.
\newblock {https://www.singular.uni-kl.de} (2024).


\bibitem{Weibel}
C.\ Weibel, 
{\em An Introduction to Homological Algebra}. Cambridge Studies in Advanced Mathematics, Cambridge University Press, Cambridge, 1994.

\bibitem{Zagier91} D.\ Zagier, 
``Periods of modular forms and Jacobi theta functions'',
Invent.\ math.\ 104 (1991), 449--465.

\bibitem{Indian} 
D.\ Zagier, 
``Modular forms and differential operators'',
Proc.\ Indian Acad.\ Sci.\ (Math.\ Sci.) 104(1) (1994), 57--75.

\bibitem{Zagier16}
D.\ Zagier,
``Partitions, quasimodular forms, and the Bloch--Okounkov theorem'',
Ramanujan J.\ 41(1--3) (2016),
345--368.



\bibitem{Zwegers} 
S.\ Zwegers, ``Mock Theta Functions,'' Ph.D.\ Thesis (Advisor D.\ Zagier), Universiteit Utrecht, 2002.

\end{thebibliography}


\end{document}